\documentclass[a4paper,10pt]{article}

\newtheorem{theorem}{Theorem}[section]
\newtheorem{lemma}[theorem]{Lemma}

\newtheorem{conjecture}[theorem]{Conjecture}
\newenvironment{proof}[1][Proof]{\begin{trivlist}
\item[\hskip \labelsep {\bfseries #1}]}{\end{trivlist}}
\newenvironment{definition}[1][Definition]{\begin{trivlist}
\item[\hskip \labelsep {\bfseries #1}]}{\end{trivlist}}
\newenvironment{example}[1][Example]{\begin{trivlist}
\item[\hskip \labelsep {\bfseries #1}]}{\end{trivlist}}

\newcommand{\qed}{\nobreak \ifvmode \relax \else      \ifdim\lastskip<1.5em \hskip-\lastskip
\hskip1.5em plus0em minus0.5em \fi \nobreak
\vrule height0.75em width0.5em depth0.25em\fi}

\usepackage{amssymb}
\usepackage{mathrsfs}

\begin{document}

\begin{center}

{\huge Differences of Skew Schur Functions of Staircases with Transposed Foundations}

\

{\LARGE Matthew Morin}

University of British Columbia

mjmorin@math.ubc.ca
\

\end{center}



\

\begin{center}
\textbf{Abstract}
\end{center}

We consider the skew diagram $\Delta_n$, which is the $180^\circ$ rotation of the staircase diagram $\delta_n = (n,n-1,n-2,\ldots,2,1)$. 
We create a \textit{staircase with bad foundation} by augmenting $\Delta_n$ with another skew diagram, which we call the \textit{foundation}. 
We consider pairs of staircases with bad foundation whose foundations are transposes of one another. 
Among these pairs, we show that the difference of the corresponding skew Schur functions is Schur-positive in the case when one of the foundations consists of either a one or two row diagram, or a hook diagram.

\setlength{\unitlength}{0.1mm}

\newcommand{\balpha}{ \alpha^{\triangleleft} }

\newcommand{\bbeta}{ \beta^{\triangleleft}}

\newcommand{\lroof}{\mbox{\begin{picture}(3,0)
                                \put(0,0){\line(0,1){10}}

                                \put(0,10){\line(1,0){3}}

                           \end{picture}}}

\newcommand{\rroof}{\mbox{\begin{picture}(3,0)
                                \put(3,0){\line(0,1){10}}

                                \put(3,10){\line(-1,0){3}}

                           \end{picture}}}

\newcommand{\longlroof}{\mbox{\begin{picture}(3,0)
                                \put(0,-4){\line(0,1){14}}

                                \put(0,10){\line(1,0){3}}

                           \end{picture}}}

\newcommand{\longrroof}{\mbox{\begin{picture}(3,0)
                                \put(3,-4){\line(0,1){14}}

                                \put(3,10){\line(-1,0){3}}

                           \end{picture}}}

\newcommand{\wedgeline}{\mbox{\begin{picture}(0,0)
                                \put(-1.5,0){$\wedge$}
                                \put(-2,3){\line(0,-1){3}}
                              \end{picture}}}
\newcommand{\Boox}{\mbox{\begin{picture}(0,0)
                                \put(0,-3){ \framebox(10,10){\tiny{$j-1$}} }
                            \end{picture}}}

\newcommand{\hooka}{\mbox{\begin{picture}(0,0)
                                \put(-5,-3){ \framebox(10,10){} }
                                \put(-5,-13){ \framebox(10,10){} }
                                \put(-5,-23){ \framebox(10,10){} }
                                \put(-5,-33){ \framebox(10,10){} }
                                \put(-5,-43){ \framebox(10,10){} }
                                \put(-5,-53){ \framebox(10,10){} }
                            \end{picture}}}

\newcommand{\hookb}{\mbox{\begin{picture}(0,0)
                                \put(-5,-3){ \framebox(10,10){} }
                                \put(5,-3){ \framebox(10,10){} }
                                \put(-5,-13){ \framebox(10,10){} }
                                \put(-5,-23){ \framebox(10,10){} }
                                \put(-5,-33){ \framebox(10,10){} }
                                \put(-5,-43){ \framebox(10,10){} }
                            \end{picture}}}

\newcommand{\hookc}{\mbox{\begin{picture}(0,0)
                                \put(-5,-3){ \framebox(10,10){} }
                                \put(5,-3){ \framebox(10,10){} }
                                \put(15,-3){ \framebox(10,10){} }
                                \put(-5,-13){ \framebox(10,10){} }
                                \put(-5,-23){ \framebox(10,10){} }
                                \put(-5,-33){ \framebox(10,10){} }
                            \end{picture}}}

\newcommand{\hookd}{\mbox{\begin{picture}(0,0)
                                \put(-5,-3){ \framebox(10,10){} }
                                \put(5,-3){ \framebox(10,10){} }
                                \put(15,-3){ \framebox(10,10){} }
                                \put(25,-3){ \framebox(10,10){} }
                                \put(-5,-13){ \framebox(10,10){} }
                                \put(-5,-23){ \framebox(10,10){} }
                            \end{picture}}}

\newcommand{\hooke}{\mbox{\begin{picture}(0,0)
                                \put(-5,-3){ \framebox(10,10){} }
                                \put(5,-3){ \framebox(10,10){} }
                                \put(15,-3){ \framebox(10,10){} }
                                \put(25,-3){ \framebox(10,10){} }
                                \put(35,-3){ \framebox(10,10){} }
                                \put(-5,-13){ \framebox(10,10){} }
                            \end{picture}}}

\newcommand{\hookf}{\mbox{\begin{picture}(0,0)
                                \put(-5,-3){ \framebox(10,10){} }
                                \put(5,-3){ \framebox(10,10){} }
                                \put(15,-3){ \framebox(10,10){} }
                                \put(25,-3){ \framebox(10,10){} }
                                \put(35,-3){ \framebox(10,10){} }
                                \put(45,-3){ \framebox(10,10){} }
                            \end{picture}}}

\setlength{\unitlength}{0.1mm}

\newcommand{\starbox}{\mbox{\begin{picture}(0,0)
                                \put(0,0){ \framebox(10,10){ }}
                                \put(6,2){$*$}
                           \end{picture}}}

\newcommand{\bbox}{\mbox{\begin{picture}(0,0)
                                \put(0,0){ \framebox(10,10){ }}
                                \put(6,2){$b$}
                           \end{picture}}}

\newcommand{\cbox}{\mbox{\begin{picture}(0,0)
                                \put(0,0){ \framebox(10,10){ }}
                                \put(6,2){$c$}
                           \end{picture}}}

\newcommand{\Done}{\mbox{\begin{picture}(0,0)
                                \put(-20,-23){ \framebox(10,10){ }}
                                \put(-10,-13){ \framebox(10,10){ }}
                                \put(-10,-3){ \framebox(10,10){ }}
                                \put(-10,-23){ \framebox(10,10){ }}
                            \end{picture}}}

\newcommand{\Dtwo}{\mbox{\begin{picture}(0,0)
                                \put(-20,-23){ \framebox(10,10){ }}
                                \put(-10,-23){ \framebox(10,10){ }}
                                \put(-20,-13){ \framebox(10,10){ }}
                                \put(-10,-13){ \framebox(10,10){ }}
                                \put(0,-13){ \framebox(10,10){ }}
                                \put(10,-13){ \framebox(10,10){ }}
                                \put(-20,-3){ \framebox(10,10){ }}
                                \put(-10,-3){ \framebox(10,10){ }}
                                \put(0,-3){ \framebox(10,10){ }}
                                \put(10,-3){ \framebox(10,10){ }}
                             \end{picture}}}

\newcommand{\Donedottwo}{\mbox{\begin{picture}(0,0)
                                \put(-20,-23){ \framebox(10,10){ }}
                                \put(-10,-23){ \framebox(10,10){ }}
                                \put(-20,-13){ \framebox(10,10){ }}
                                \put(-10,-13){ \framebox(10,10){ }}
                                \put(0,-13){ \framebox(10,10){ }}
                                \put(10,-13){ \framebox(10,10){ }}
                                \put(-20,-3){ \framebox(10,10){ }}
                                \put(-10,-3){ \framebox(10,10){ }}
                                \put(0,-3){ \framebox(10,10){ }}
                                \put(10,-3){ \framebox(10,10){ }}

                                \put(-30,-53){ \framebox(10,10){ }}
                                \put(-20,-43){ \framebox(10,10){ }}
                                \put(-20,-33){ \framebox(10,10){ }}
                                \put(-20,-53){ \framebox(10,10){ }}
                             \end{picture}}}

\newcommand{\Doneodottwo}{\mbox{\begin{picture}(0,0)
                                \put(-20,-23){ \framebox(10,10){ }}
                                \put(-10,-23){ \framebox(10,10){ }}
                                \put(-20,-13){ \framebox(10,10){ }}
                                \put(-10,-13){ \framebox(10,10){ }}
                                \put(0,-13){ \framebox(10,10){ }}
                                \put(10,-13){ \framebox(10,10){ }}
                                \put(-20,-3){ \framebox(10,10){ }}
                                \put(-10,-3){ \framebox(10,10){ }}
                                \put(0,-3){ \framebox(10,10){ }}
                                \put(10,-3){ \framebox(10,10){ }}

                                \put(-40,-43){ \framebox(10,10){ }}
                                \put(-30,-33){ \framebox(10,10){ }}
                                \put(-30,-23){ \framebox(10,10){ }}
                                \put(-30,-43){ \framebox(10,10){ }}
                             \end{picture}}}

\section{Introduction}

The development and study of Schur functions has shown them to not only be integral to the study of symmetric functions, but also to have many interesting and far-reaching consequences. 
The Schur functions are perhaps best known as a basis of the ring of symmetric functions. As such, the structure constants $c_{\mu \nu}^{\lambda}$, commonly known as the \textit{Littlewood-Richardson coefficients}, appearing in the products
\[ s_\mu s_\nu = \sum_{\lambda} c_{\mu \nu}^{\lambda} s_{\lambda}, \]
are of paramount importance to the ring structure.

The Schur functions also appear in the study of representation theorey. 
In particular, in representations of the symmetric group, general linear groups, and special linear groups. 
For instance, in representations of the symmetric group, the Specht modules are in one-to-one correspondence with the Schur functions. 
Further, given two Specht modules $S^\mu$ and $S^\nu$ we have
\[ (S^\mu \otimes S^\nu)\uparrow^{S_n} = \bigoplus_{\lambda} c_{\mu \nu}^{\lambda} S^{\lambda}. \]

Schur functions are also intrinsic to the structure of the cohomology ring of the Grassmannian, where the Schubert classes are in correspondence to the Schur functions and the cup product of each pair $\sigma_\mu$, $\sigma_\nu$ of Schubert classes satisfies
\[ \sigma_\mu \cup \sigma_\nu = \sum_{\lambda} c_{\mu \nu}^{\lambda} \sigma_\lambda.\]

Besides arising in each of these areas, the Littlewood-Richardson coefficients are known to satisfy $c_{\mu \nu}^{\lambda} \geq 0.$ 
In particular, $c_{\mu \nu}^{\lambda}$ is always a non-negative integer. 
Thus each product $s_\mu s_\nu$ gives rise to a linear combination of Schur functions with non-negative coefficients. 
We call such expressions \textit{Schur-positive}.
Another famous example of Schur-positivity arises when considering skew Schur functions. 
Namely, we have 
\[ s_{\lambda / \mu} = \sum_{\nu} c_{\mu \nu}^{\lambda} s_{\nu},\]
for each skew diagram $\lambda / \mu$.

In recent papers, there has been considerable interest in determining the Schur-positivity of expressions of the form
\[ s_{\mu} s_{\nu} - s_{\lambda} s_{\rho} \textrm{ } \textrm{ and } \textrm{ } s_{\lambda / \mu} - s_{\rho / \nu}.\]
Each of these such Schur-positive differences gives a set of inequalities that the corresponding Littlewood-Richardson coefficients must satisfy. 
However, we note that each product of Schur functions $s_{\mu} s_{\nu}$ can be realized as a certain skew Schur function, we can focus on the problem of skew Schur differences. 
We shall construct certain families of skew Schur functions and completely describe all the Schur-positive differences within each family.

\section{Preliminaries}

We begin by briefly introducing the various objects, notations, and results we shall require.  
A complete study can be found in such sources such as \cite{sagan} of \cite{stanley}.

A \textit{partition} $\lambda$\label{def:lambda} of a positive integer $n$, written $\lambda \vdash n$, is a sequence of weakly decreasing positive integers $\lambda = (\lambda_1,\lambda_2, \ldots, \lambda_k)$ with $\sum_{i=1}^{k} \lambda_i = n$. 
We call each $\lambda_i$ a \textit{part} of $\lambda$, and if $\lambda$ has exactly $k$ parts we say $\lambda$ is of \textit{length} $k$ and write $l(\lambda)=k$.
When $\lambda \vdash n$ we will also write $| \lambda |=n$ and say that the \textit{size} of $\lambda$ is $n$. 

We shall use $j^r$ to denote the sequence $j,j,\ldots, j$ consisting of $r$ $j$'s.
Thus, we shall write $\lambda = (k^{r_k}, {k-1}^{r_{k-1}}, \ldots ,1^{r_1})$ \label{parts} for the partition which has $r_1$ parts of size one, $r_2$ parts of size two, \ldots,  and $r_k$ parts of size $k$.

We say $\alpha=(\alpha_1,\alpha_2,\ldots,\alpha_k)$ is a \label{comp} \textit{composition} of $n$ if each $\alpha_i$ is a positive integer and $\sum_{i=1}^{k} \alpha_i = n$.
As with partitions, we call each $\alpha_i$ a \textit{part} of $\alpha$, write $|\alpha|=n$ for the \textit{size} of $\alpha$, and if $\alpha$ has exactly $k$ parts we say $\alpha$ is of \textit{length} $k$ and write $l(\alpha)=k$.
If we relax the conditions to consider sums of non-negative integers, that is, allowing some of the $\alpha_i$ to be zero, we obtain the concept of a \textit{weak composition} of $n$.

We may sometimes find it useful to treat weak compositions as vectors with non-negative integer entries.
When we write the vector ${z}=(z_1,z_2, z_3,\ldots, z_n)$ we shall mean the infinite vector ${z}=(z_1,z_2, z_3,\ldots, z_n, 0, 0, 0, \ldots)$.
However, we shall only consider vectors with finitely many non-zero entries and hence we shall only display vectors with finite length, but in this way we may unambiguously add vectors of different lengths. 
Given a positive integer $i$, we let $e_i$ denote the $i$-th standard basis vector.
That is, the vector that has its $i$-th entry equal to $1$ and all remaining entries equal to $0$.

\

Given a partition $\lambda$, we can represent it via the diagram of left justified rows of boxes whose $i$-th row contains $\lambda_i$ boxes. 
The diagrams of these type are called \textit{Ferrers diagrams}, or just \textit{diagrams} for short. 
We shall use the symbol $\lambda$ when refering to both the partition and its Ferrers diagram.

Whenever we find a diagram $D'$ contained in a diagram $D$ as a subset of boxes, we say that $D'$ is a \textit{subdiagram} of $D$.
Suppose partitions $\lambda=(\lambda_1,\lambda_2,\ldots, \lambda_j) \vdash n$ and $\mu=(\mu_1,\mu_2,\ldots, \mu_k) \vdash m$, with $m \leq n$, $k \leq j$, and $\mu_i \leq \lambda_i$ for each $i = 1, 2, \ldots, k$ are given. 
Then $\mu$ is a subdiagram of $\lambda$, and a particular copy of $\mu$ is found at the top-left corner of $\lambda$.  
We can form the \textit{skew \label{skew} diagram} $\lambda / \mu$ by removing that copy of $\mu$ from $\lambda$. 
Henceforth, when we say that $D$ is a \textit{diagram}, it is assumed that $D$ is either the Ferrers diagram of some partition $\lambda$, or $D$ is the skew diagram $\lambda / \mu$ for some partitions $\lambda, \mu$. 

\begin{example}
Here we consider the partitions $\lambda=(4,4,2)$ and $\mu=(3,1)$, and form the skew diagram $\lambda / \mu$.

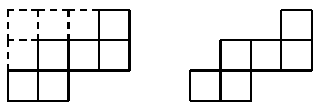
\begin{center}
\setlength{\unitlength}{0.4mm}
\begin{picture}(100,40)(-30,-5)

   \put(-30,0){\line(1,0){20}}
   \put(-30,10){\line(1,0){40}}
   \put(-20,20){\line(1,0){30}}
   \put(0,30){\line(1,0){10}}

   \put(-30,0){\line(0,1){10}}
   \put(-20,0){\line(0,1){20}}
   \put(-10,0){\line(0,1){20}}
   \put(0,10){\line(0,1){20}}
   \put(10,30){\line(0,-1){20}}

\put(-30,10){\dashbox{2}(10,10){}}
\put(-30,20){\dashbox{2}(10,10){}}
\put(-30,20){\dashbox{2}(10,10){}}
\put(-20,20){\dashbox{2}(10,10){}}
\put(-10,20){\dashbox{2}(10,10){}}

   \put(30,0){\line(1,0){20}}
   \put(30,10){\line(1,0){40}}
   \put(40,20){\line(1,0){30}}
   \put(60,30){\line(1,0){10}}

   \put(30,0){\line(0,1){10}}
   \put(40,0){\line(0,1){20}}
   \put(50,0){\line(0,1){20}}
   \put(60,10){\line(0,1){20}}
   \put(70,30){\line(0,-1){20}}

\end{picture}
\end{center}

\end{example}

The number of boxes that appears in a given row or a given column of a diagram is called the \textit{length} of that row or column. 
Given any diagram $D$, the $180^{\circ}$ rotation of a diagram diagram $D$ is denoted by $D^{\circ}$.

For two boxes $b_1$ and $b_2$ in a diagram $D$, we define a \textit{path} from $b_1$ to $b_2$ in $D$ to be a sequence of steps either up, down, left, or right that begins at $b_1$, ends at $b_2$, and at no time leaves the diagram $D$.
We say that a diagram $D$ is \textit{connected} if for any two boxes $b_1$ and $b_2$ of $D$ there is a path from $b_1$ to $b_2$ in $D$. If $D$ is not connected we say it is \textit{disconnected}.

\

A \textit{hook} is the Ferrers diagram corresponding to a partition $\lambda$ that satisfies $\lambda_i \leq 1$ for all $i>1$. 
Hence a hook has at most one row of length larger than $1$.

\begin{example}
Here we see the hooks $\lambda = (4,1,1)$ and $\mu = (5,1)$.

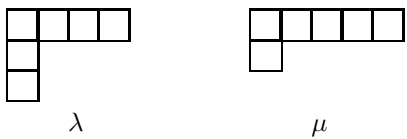
\begin{center}
\setlength{\unitlength}{0.4mm}
\begin{picture}(100,40)(-30,-5)

\put(-20,-10){$\lambda$}

\put(-10,20){\framebox(10,10)[tl]{ }}
\put(-20,20){\framebox(10,10)[tl]{ }}
\put(-30,20){\framebox(10,10)[tl]{ }}
\put(-40,20){\framebox(10,10)[tl]{ }}
\put(-40,10){\framebox(10,10)[tl]{ }}
\put(-40,0){\framebox(10,10)[tl]{ }}

\put(60,-10){$\mu$}

\put(80,20){\framebox(10,10)[tl]{ }}
\put(70,20){\framebox(10,10)[tl]{ }}
\put(60,20){\framebox(10,10)[tl]{ }}
\put(50,20){\framebox(10,10)[tl]{ }}
\put(40,20){\framebox(10,10)[tl]{ }}
\put(40,10){\framebox(10,10)[tl]{ }}

\end{picture}
\end{center}

\end{example}

\

If $D$ is a diagram, then a \textit{tableau}---plural \textit{tableaux}---$\mathcal{T}$ \textit{of shape $D$} is the array obtained by filling the boxes of the $D$ with the positive integers, where repetition is allowed. 
A tableau is said to be a \textit{semistandard Young tableau} (or simply \textit{semistandard}, for short) if each row gives a weakly increasing sequence of integers and each columns gives a strictly increasing sequence of integers.
We will often abbreviate ``semistandard Young tableau" by SSYT and ``semistandard Young tableaux" by SSYTx. 

When we wish to depict a certain tableau we will either show the underlying diagram with the entries residing in the boxes of the diagram or we may simply replace the boxes with the entries, so that the entries themselves depict the underlying shape of the tableau.

The \textit{content} of a tableau $\mathcal{T}$ is the weak composition given by
\[ \nu(\mathcal{T}) = ( \# \textrm{1's in }\mathcal{T}, \# \textrm{2's in }\mathcal{T}, \ldots).\]

Now, given a tableau $\mathcal{T}$ of shape $D$, we may wish to focus on the entries of $\textit{T}$ that lie in some subdiagram $D'$ of $D$. 
In this way we obtain a \textit{subtableau} of shape $D'$.
Similarly, given a tableau $\mathcal{T}'$ of shape $D'$ and a diagram $D$ containing $D'$, we may sometimes wish to fill the copy of $D'$ within $D$ as in the tableau $\mathcal{T}'$. 

\

Given a skew diagram $D$, the \textit{skew Schur symmetric function corresponding to $D$} is defined to be \label{slambda}
\begin{equation}
\label{schurdef}
 s_{D}(\textbf{x}) = \sum_{\mathcal{T}} {x_1}^{\# \textrm{1's in } \mathcal{T}}{x_2}^{\# \textrm{2's in } \mathcal{T}} \cdots,
\end{equation}
where the sum is taken over all semistandard Young tableaux $\mathcal{T}$ of shape $D$. 
When $D=\lambda$ is a partition, $s_\lambda$ is called the \textit{Schur symmetric function corresponding to $\lambda$}.

The set $\{s_\lambda | \lambda \vdash n \}$ is a basis of $\Lambda^n$, the set of homogeneous symmetric functions of degree $n$. 
Thus any homogeneous symmetric function of degree $n$ can be written as a unique linear combination of the $s_\lambda$. 
Therefore for each $f \in \Lambda^n$ we can write $f=\sum_{\lambda} a_\lambda s_\lambda$ for appropriate coefficients.

As was mentioned in the introduction, for any partitions $\mu$ and $\nu$, 
\begin{equation}
\label{smusnu}
 s_\mu s_\nu = \sum_{\lambda \vdash n} c_{\mu \nu}^{\lambda} s_\lambda, 
\end{equation}
and for any skew diagram $\lambda / \mu$,
\begin{equation}
\label{slambdaskewmu}
 s_{\lambda / \mu} = \sum_{\nu \vdash n} c_{\mu \nu}^{\lambda} s_\nu 
\end{equation}
where the $c_{\mu \nu}^{\lambda}$ are the \textit{Littlewood-Richardson coefficients}.
Further, it turns out that the Littlewood-Richardson coefficients are non-negative integers and they count an interesting class of SSYT that we shall now describe.

Given a tableau $\mathcal{T}$, the \textit{reading word} of $\mathcal{T}$ is the sequence of integers obtained by reading the entries of the rows of $\mathcal{T}$ from right to left, proceeding from the top row to the bottom. 
We say that a sequence $r=r_1,r_2,\ldots, r_k$ is \textit{lattice} if, for each $j$, when reading the sequence from left to right the number of $j$'s that we have read is never greater than the number of $j+1$'s that we have read.

\begin{theorem} [Littlewood-Richardson Rule] (\cite{LR})
\label{lr}

For partitions $\lambda, \mu$, and $\nu$, the \textit{Littlewood-Richardson coefficient} $c_{\mu \nu}^{\lambda}$ is the number of SSYTx of shape $\lambda / \mu$, content $\nu$, with lattice reading word.
\end{theorem}

For any $f = \sum_{\lambda \vdash n} a_{\lambda} s_{\lambda} \in \Lambda^n$, we say that $f$ is \textit{Schur-positive}, and write $f \geq_s 0$, if each $a_{\lambda} \geq 0$. 
Therefore, the Littlewood-Richardson rule shows that both $s_{\mu} s_{\nu}$ and $s_{\lambda / \mu}$ are Schur-positive for all partitions $\lambda, \mu,$ and $\nu$. 
For $f,g \in \Lambda^n$, we will be interested in whether or not the difference $f-g$ is Schur positive. 
We write $f \geq_s g$ whenever $f-g$ is Schur-positive. 
If neither $f-g$ nor $g-f$ is Schur-positive we say that $f$ and $g$ are \textit{Schur-incomparable}.

\

If we wish to consider setting each of the variables $x_{n+1}, x_{n+2}, \ldots$ to 0, we obtain the \textit{Schur polynomial} $s_{\lambda} (x_1,\ldots, x_n)$. 
These may also be obtained by considering tableaux using entries from the set $\{1,2,\ldots, n\}.$ 
The Littlewood-Richardson rule still holds in this context, with the restriction that the content have length $\leq n$. 
We shall use the notation $s_{\lambda / \mu} =_n s_{\kappa / \rho}$ to denote an equality between two skew Schur functions upon reducing to $n$ variables.

\section{Staircases with Bad Foundations}

A Ferrers diagram is a \textit{staircase} if it is the Ferrers diagram of a partition of the form $\lambda = (n,n-1,n-2,\ldots, 2,1)$ or if it is the $180^{\circ}$ rotation of such a diagram. Both these diagrams are referred to as \textit{staircases of length $n$} and will be denoted by $\delta_n$ and $\Delta_n$ respectively. For what follows we will also define $\delta_{-1}=\delta_{0} = \emptyset$.

\begin{example} Here we see the two staircases of length 5:

\

\

\

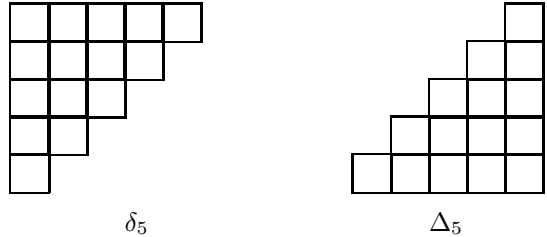
\setlength{\unitlength}{0.5mm}

\begin{picture}(100,50)(-30,-5)

\put(-10,50){\framebox(10,10)[tl]{ }}
\put(0,50){\framebox(10,10)[tl]{ }}
\put(10,50){\framebox(10,10)[tl]{ }}
\put(20,50){\framebox(10,10)[tl]{ }}
\put(30,50){\framebox(10,10)[tl]{ }}
     
\put(-10,40){\framebox(10,10)[tl]{ }}
\put(0,40){\framebox(10,10)[tl]{ }}
\put(10,40){\framebox(10,10)[tl]{ }}
\put(20,40){\framebox(10,10)[tl]{ }}
 
\put(-10,30){\framebox(10,10)[tl]{ }}
\put(0,30){\framebox(10,10)[tl]{ }}
\put(10,30){\framebox(10,10)[tl]{ }}

\put(-10,20){\framebox(10,10)[tl]{ }}
\put(0,20){\framebox(10,10)[tl]{ }}

\put(-10,10){\framebox(10,10)[tl]{ }}

\put(20,0){$\delta_5$}

\put(100,0){$\Delta_5$}

\put(120,50){\framebox(10,10)[tl]{ }}
\put(120,40){\framebox(10,10)[tl]{ }}
\put(120,30){\framebox(10,10)[tl]{ }}
\put(120,20){\framebox(10,10)[tl]{ }}
\put(120,10){\framebox(10,10)[tl]{ }}

\put(110,40){\framebox(10,10)[tl]{ }}
\put(110,30){\framebox(10,10)[tl]{ }}
\put(110,20){\framebox(10,10)[tl]{ }}
\put(110,10){\framebox(10,10)[tl]{ }}

\put(100,30){\framebox(10,10)[tl]{ }}
\put(100,20){\framebox(10,10)[tl]{ }}
\put(100,10){\framebox(10,10)[tl]{ }}

\put(90,20){\framebox(10,10)[tl]{ }}
\put(90,10){\framebox(10,10)[tl]{ }}

\put(80,10){\framebox(10,10)[tl]{ }}

\end{picture}

\end{example}

If $\delta_n \subseteq \lambda$ as Ferrers diagrams then we say that \textit{$\lambda$ contains a staircase of size $n$}. In these cases we may create the skew diagram $\lambda / \delta_n$.

\begin{definition} Given $n \geq 1$, $k \geq 0$, and a partition $\lambda$ containing a staircase $\delta_{k-1}$ such that skew diagram $\lambda / \delta_{k-1}$ is connected and $k \leq \lambda_1 \leq n+k$, then we can create a skew diagram called a \textit{staircase with a bad foundation}, denoted $\mathcal{S}(\lambda, k,n)$, by placing $\lambda / \delta_{k-1}$ immediately below $\Delta_n$ such that the rows of the two diagrams overlap in precisely $\lambda_1-k$ positions. We call the subdiagram $\lambda / \delta_{k-1}$ the \textit{foundation} of $\mathcal{S}(\lambda, k,n)$.
\end{definition}

In the staircase with bad foundation $\mathcal{S}(\lambda, k,n)$, $k$ is the distance to the left of $\Delta_n$ that $\lambda$ extends before the copy of $\delta_{k-1}$ is removed. 
The requirement $k \leq \lambda_1$, guarantees that the top row of the skew diagram $\lambda / \delta_{k-1}$ is non-emtpy. 
The diagram $\mathcal{S}(\lambda, k,n)$ is connected if $k<\lambda_1$ and has two connected components, $\lambda / \delta_{k-1}$ and $\Delta_n$, if $k=\lambda_1$.
We also note that the skew diagram $\mathcal{S}(\lambda, k,n)$ can be written in the form $\mu / \delta_{n+k-1}$, where $\mu$ is a partition.

\begin{example}

\

For $k=0$ we have $\delta_{k-1}= \delta_{-1} = \emptyset$, and $\Delta_{n}$ and $\lambda$ overlap in $\lambda_1$ places. Thus $\mathcal{S}(\lambda, k,n)$ consists of the foundation $\lambda$ left-justified with the staircase $\Delta_n$. For example, with $n=5$, $k=0$, and $\lambda=(5,4,2)$ we have the following staircase with bad foundation:

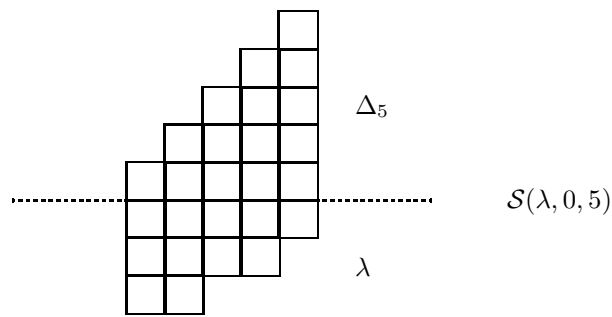
\setlength{\unitlength}{0.5mm}

\begin{picture}(100,50)(-30,-5)

\put(-10,-10){\dashbox{1}(110,0)[tl]{ }}

\put(80,13){$\Delta_5$}

\put(20,-10){\framebox(10,10)[tl]{ }}
\put(30,-10){\framebox(10,10)[tl]{ }}
\put(40,-10){\framebox(10,10)[tl]{ }}
\put(50,-10){\framebox(10,10)[tl]{ }}
\put(60,-10){\framebox(10,10)[tl]{ }}
     
\put(30,0){\framebox(10,10)[tl]{ }}
\put(40,0){\framebox(10,10)[tl]{ }}
\put(50,0){\framebox(10,10)[tl]{ }}
\put(60,0){\framebox(10,10)[tl]{ }}
 
\put(40,10){\framebox(10,10)[tl]{ }}
\put(50,10){\framebox(10,10)[tl]{ }}
\put(60,10){\framebox(10,10)[tl]{ }}

\put(50,20){\framebox(10,10)[tl]{ }}
\put(60,20){\framebox(10,10)[tl]{ }}

\put(60,30){\framebox(10,10)[tl]{ }}

\put(20,-20){\framebox(10,10)[tl]{ }}
\put(30,-20){\framebox(10,10)[tl]{ }}
\put(40,-20){\framebox(10,10)[tl]{ }}
\put(50,-20){\framebox(10,10)[tl]{ }}
\put(60,-20){\framebox(10,10)[tl]{ }}

\put(20,-30){\framebox(10,10)[tl]{ }}
\put(30,-30){\framebox(10,10)[tl]{ }}
\put(40,-30){\framebox(10,10)[tl]{ }}
\put(50,-30){\framebox(10,10)[tl]{ }}

\put(20,-40){\framebox(10,10)[tl]{ }}
\put(30,-40){\framebox(10,10)[tl]{ }}

\put(80,-30){$\lambda$}

\put(120,-12){$\mathcal{S}(\lambda,0,5)$}

\end{picture}

\

\

\

\

\

For $k=1$ we have $\delta_{k-1}= \delta_{0} = \emptyset$, and $\Delta_{n}$ and $\lambda$ overlap in $\lambda_1-1$ places. Thus $\mathcal{S}(\lambda, k,n)$ is the diagram $\mathcal{S}(\lambda, 0,n)$ with its foundation $\lambda$ shifted one to the left. For example, with $n=5$, $k=1$, and $\lambda=(5,4,2)$ we have the following staircase with bad foundation:

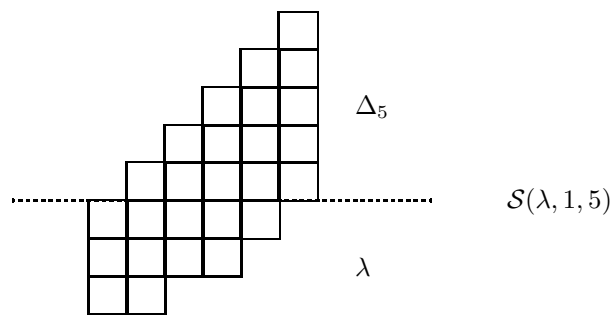
\setlength{\unitlength}{0.5mm}

\begin{picture}(100,50)(-30,-5)

\put(-10,-10){\dashbox{1}(110,0)[tl]{ }}

\put(80,13){$\Delta_5$}

\put(20,-10){\framebox(10,10)[tl]{ }}
\put(30,-10){\framebox(10,10)[tl]{ }}
\put(40,-10){\framebox(10,10)[tl]{ }}
\put(50,-10){\framebox(10,10)[tl]{ }}
\put(60,-10){\framebox(10,10)[tl]{ }}
     
\put(30,0){\framebox(10,10)[tl]{ }}
\put(40,0){\framebox(10,10)[tl]{ }}
\put(50,0){\framebox(10,10)[tl]{ }}
\put(60,0){\framebox(10,10)[tl]{ }}
 
\put(40,10){\framebox(10,10)[tl]{ }}
\put(50,10){\framebox(10,10)[tl]{ }}
\put(60,10){\framebox(10,10)[tl]{ }}

\put(50,20){\framebox(10,10)[tl]{ }}
\put(60,20){\framebox(10,10)[tl]{ }}

\put(60,30){\framebox(10,10)[tl]{ }}

\put(10,-20){\framebox(10,10)[tl]{ }}
\put(20,-20){\framebox(10,10)[tl]{ }}
\put(30,-20){\framebox(10,10)[tl]{ }}
\put(40,-20){\framebox(10,10)[tl]{ }}
\put(50,-20){\framebox(10,10)[tl]{ }}

\put(10,-30){\framebox(10,10)[tl]{ }}
\put(20,-30){\framebox(10,10)[tl]{ }}
\put(30,-30){\framebox(10,10)[tl]{ }}
\put(40,-30){\framebox(10,10)[tl]{ }}

\put(10,-40){\framebox(10,10)[tl]{ }}
\put(20,-40){\framebox(10,10)[tl]{ }}

\put(80,-30){$\lambda$}

\put(120,-12){$\mathcal{S}(\lambda,1,5)$}

\end{picture}

\end{example}

\

\

\

\

Another advantage in computing the skew Schur function of a staircase with bad foundation is the following fact, which arises from the structure of the unique content, $\nu = \delta_n = (n,n-1,\ldots,2,1)$, just mentioned.

\begin{lemma}
\label{firstrowlem}
Let $\mathcal{S}(\lambda,k,n)$ be a staircase with bad foundation and $\mathcal{T}$ be a SSYT of shape $\mathcal{S}(\lambda,k,n)$ whose reading word is lattice.
Then the entries in the first row of the foundation $\lambda / \delta_{k-1}$ of $\mathcal{T}$ form a strictly increasing subsequence of the set 
\[ \{2, 3, \ldots, n+1 \} \cup \left\{ \begin{array}{rll}
                                       \{1\} & \textrm{if} & k >0 \\
                                       \emptyset & \textrm{if} & k =0 \\
\end{array} \right.  .\]
\end{lemma}

\begin{proof}
Let $\mathcal{T}$ be a SSYT of shape $\mathcal{S}(\lambda,k,n)$ whose reading word is lattice.
As discussed previously, the Littlewood-Richardson rule allows only one way of filling the $\Delta_n$ portion of the shape $\mathcal{S}(\lambda,k,n)$. 
The content of this filling of $\Delta_n$ is $(n,n-1,\ldots,2,1)$.

Let $R$ be the first row of the foundation, and suppose that $R$ contains the value $j$ twice. 
Then $j \neq 1$ since the second value in $R$ is immediately below a value $\geq 1$ and the columns of $\mathcal{T}$ strictly increase.
Also, if $j>1$, then when reading $\mathcal{T}$ the lattice condition will be violated once we have read both $j$'s.
Therefore no value $j$ can be repeated in $R$, and hence the first row of the foundation of $\mathcal{T}$ contains no repeated values. 

Now, the lattice condidion prohibits any value greater than $n+1$ from appearing in $R$.
Also, if the value $1$ appears in $R$ then it must be the first entry of $R$. 
If $k=0$, then this entry would be below an entry of the staircase $\Delta_n$ and the columns would not strictly increase. 
Hence the entry $1$ can appear in $R$ only if $k \geq 1$. \qed

\end{proof}

For the next result we only allow $0 \leq k \leq 1$. In these cases we have $\delta_{k-1} = \emptyset$. Thus the foundation of $\mathcal{S}(\lambda,k,n)$ is simply $\lambda$.

\begin{lemma} 
\label{fatjoinlemma}
Let $k$ and $n$ be given with $0 \leq k \leq 1$,  and a partition $\lambda$ such that $\lambda_1 \leq n+k$ be given.
If $T$ is a SSYT of shape $\lambda  \oplus \Delta_n$ with lattice reading word such that there is at most $k$ $1$'s in the first row of $\lambda$, then the tableau of shape $\mathcal{S}(\lambda,k,n)$ obtained from $T$ by shifting the foundation $\lambda$ to the right is also a SSYT with lattice reading word.

\end{lemma}

\begin{proof}
Let $T$ be a SSYT of shape $\lambda   \oplus \Delta_n$ with lattice reading word and let $T_k'$ be the tableau of shape $\mathcal{S}(\lambda,k,n)$ obtained from $T$ by shifting the foundation $\lambda$ to the right.
Since shifting $\lambda $ to the right does not affect the order in which the entries are read, $T_k'$ has a lattice reading word.
Also, the rows of $T_k'$ weakly increase since they are the same as the rows of $T$.
Further, to check that the columns of $T_k'$ strictly increase, we need only check that they strictly increase at the positions where the two subdiagrams $\Delta_{n}$ and $\lambda$ are joined.

The entries in the last row of $\Delta_n$ are $1,2,\ldots, n$.

As in Lemma~\ref{firstrowlem}, the entries of the first row of $\lambda$ are distinct. 
Let $r_1 > r_2 > \ldots$ be the entries of the first row of $\lambda$. 
Since there is at most $k$ $1$'s in the first row of $\lambda$, we have $r_1 \geq 1$ if $k = 1$ and $r_1 \geq 2$ if $k = 0$.

Suppose $k=1$. 
Then since $r_1 > r_2 > \ldots$ and $r_1 \geq 1$, we have $r_i>i$ for each $i$. 
Also $r_1$ is below no entry of $\Delta_n$ and, for each $i>1$, $r_i$ is below $i-1$ in $\Delta_n$. 
Therefore the columns of $T_k'$ strictly increase.

Suppose $k=0$.
Then since $r_1 > r_2 > \ldots$ and $r_1 \geq 2$, we have $r_i> i+1$ for each $i$. 
Also, for each $i \geq 1$, $r_i$ is below $i$ in $\Delta_n$. 
Therefore the columns of $T_k'$ strictly increase. \qed

\end{proof}

\section{Differences of Transposed Foundations}

 \addtocounter{page}{-1}

We shall look at differences of skew Schur functions of the form $s_{\mathcal{S}(\lambda^{t}, k,n)}-s_{\mathcal{S}(\lambda,k,n)}$, where $\lambda$ is restricted to be a partition with one part, a partition with two parts, or a hook partition. 
In both cases the difference is seen to be Schur-positive. 
Further, we find a simple expression for the difference in the case when $\lambda$ has a single part, from which we can see that the difference is multiplicity-free.

\subsection{Single Part Partitions}

\begin{theorem}
\label{1theorem}
Let $n$ be a positive integer, $0 \leq k \leq 1$, and $\lambda$ is a partition with a single part such that $\lambda_1 \leq k+n$. Then 
\begin{enumerate}
\item $s_{\mathcal{S}(\lambda^{t}, k,n)}-s_{\mathcal{S}(\lambda,k,n)} \geq_s 0,$ 
\item $s_{\mathcal{S}(\lambda^{t}, k,n)}=_{n+1}s_{\mathcal{S}(\lambda,k,n)},$ and
\item $s_{\mathcal{S}(\lambda^{t}, k,n)} \neq_{n+2} s_{\mathcal{S}(\lambda,k,n)}$ if $\lambda \neq \lambda^{t}$. 
\end{enumerate}
Furthermore, we have 
\[s_{\mathcal{S}(\lambda^{t}, k,n)}-s_{\mathcal{S}(\lambda,k,n)} = 
\displaystyle\sum_{ \stackrel{\mbox{\scriptsize $A \subseteq \{ 2-k,3-k,\ldots, n \}$}}{|A|\leq \lambda_1-2}}   s_{\nu(A)}, \]
where $\nu(A) = \delta_n + e_A + (0^n,1^{\lambda_1-|A|} )$.  
In particular, $s_{\mathcal{S}(\lambda^{t}, k,n)}-s_{\mathcal{S}(\lambda,k,n)}$ is multiplicity-free.

\end{theorem}

\begin{example}
Let $n=4$, $\lambda = (3)$, and $k=0$.
We are interested in the following two diagrams.

\setlength{\unitlength}{0.4mm}

\begin{picture}(0,100)(-30,10)

\put(70,90){\framebox(10,10)[tl]{ }}

\put(60,80){\framebox(10,10)[tl]{ }}
\put(70,80){\framebox(10,10)[tl]{ }}

\put(50,70){\framebox(10,10)[tl]{ }}
\put(60,70){\framebox(10,10)[tl]{ }}
\put(70,70){\framebox(10,10)[tl]{ }}

\put(40,60){\framebox(10,10)[tl]{ }}
\put(50,60){\framebox(10,10)[tl]{ }}
\put(60,60){\framebox(10,10)[tl]{ }}
\put(70,60){\framebox(10,10)[tl]{ }}

\put(40,50){\framebox(10,10)[tl]{ }}
\put(40,40){\framebox(10,10)[tl]{ }}
\put(40,30){\framebox(10,10)[tl]{ }}

\put(40,15){$\mathcal{S}( \lambda^t , 0,4)$}
\put(140,15){$\mathcal{S}( \lambda , 0,4)$}

\put(170,90){\framebox(10,10)[tl]{ }}

\put(160,80){\framebox(10,10)[tl]{ }}
\put(170,80){\framebox(10,10)[tl]{ }}

\put(150,70){\framebox(10,10)[tl]{ }}
\put(160,70){\framebox(10,10)[tl]{ }}
\put(170,70){\framebox(10,10)[tl]{ }}

\put(140,60){\framebox(10,10)[tl]{ }}
\put(150,60){\framebox(10,10)[tl]{ }}
\put(160,60){\framebox(10,10)[tl]{ }}
\put(170,60){\framebox(10,10)[tl]{ }}

\put(140,50){\framebox(10,10)[tl]{ }}
\put(150,50){\framebox(10,10)[tl]{ }}
\put(160,50){\framebox(10,10)[tl]{ }}

\end{picture}

\noindent Theorem~\ref{1theorem} gives
\begin{eqnarray*}
s_{\mathcal{S}( \lambda^t , 0,4)}-s_{\mathcal{S}( \lambda,0,4)} &=& \displaystyle\sum_{ \stackrel{\mbox{\scriptsize $A \subseteq \{ 2,3, 4 \}$}}{|A|\leq 1}}   s_{\nu(A)} \\
&=& s_{\nu(\emptyset)} +  s_{\nu(\{2\})} +  s_{\nu(\{3\})} +  s_{\nu(\{4\})} \\
&=& s_{(4,3,2,1,1,1,1)} + s_{(4,4,2,1,1,1)} + s_{(4,3,3,1,1,1)} + s_{(4,3,2,2,1,1)}, \\
\end{eqnarray*}
where we have omitted the details of computing $\nu(A)$ for the various $A$.
For an example of this computation, we have 
\begin{eqnarray*}
\nu(\{2\}) &=& \delta_4 + \sum_{a \in \{2 \} } e_a + (0^4,1^{3-1} ) \\
&=& (4,3,2,1) + (0,1,0,0) + (0,0,0,0,1,1) \\
&=& (4,4,2,1,1,1). \\
\end{eqnarray*}

\end{example}

\begin{proof} (of Theorem~\ref{1theorem})

\

\noindent \textbf{1.}

To prove the Schur-positivity of $s_{\mathcal{S}(\lambda^{t}, k,n)}-s_{\mathcal{S}(\lambda,k,n)}$ we show that any SSYT $\mathcal{T}_1$ of shape $\mathcal{S}(\lambda,k,n)$ with lattice reading word gives rise to a SSYT $\mathcal{T}_2$ of shape $\mathcal{S}(\lambda^{t},k,n)$ with lattice reading word and the same content. 
We then show that we can recover $\mathcal{T}_1$ from $\mathcal{T}_2$.

Consider any SSYT $\mathcal{T}_1$ of shape $\mathcal{S}(\lambda,k,n)$ with lattice reading word. 
Lemma~\ref{firstrowlem} implies that the foundation $\lambda$ contains a strictly increasing sequence $a_1<a_2<\ldots < a_{\lambda_1}$ where $a_1 \geq 2-k$ and $a_{\lambda_1} \leq n+1$. 

Since $a_1<a_2<\ldots < a_{\lambda_1}$, we can create a SSYT of shape $\lambda^{t}$ by transposing the entries of the foundation $\lambda$. 
Thus we have a SSYT of shape $\lambda \oplus \Delta_n$. 
Suppose the tableau did not have a lattice reading word. Then, since the content of $\Delta_n$ was $(n,n-1,\ldots, 3,2,1)$ and $a_i \leq n+1$ for each $i$, we would require two of the $a_i$ to be equal. Since the $a_i$ are distinct, this tableau does have a lattice reading word.
Then, we may apply Lemma~\ref{fatjoinlemma}, we obtain a SSYT $\mathcal{T}_2$ of shape $\mathcal{S}(\lambda^t,k,n)$ with lattice reading word. 
Furthermore, we may recover $\mathcal{T}_1$ from $\mathcal{T}_2$ by transposing the entries of the foundation $\lambda^{t}$ of $\mathcal{T}_2$.

This completes the proof that $s_{\mathcal{S}(\lambda^{t}, k,n)}-s_{\mathcal{S}(\lambda,k,n)} \geq_s 0$.

\

\noindent \textbf{2.}

As discussed in the introduction, using the Littlewood-Richardson rule to compute the skew Schur functions in $n+1$ variables reduces the usual Littlewood-Richardson rule to considering only those fillings that use numbers from the set $\{1,2,\ldots, n+1\}$.
We now show that we can reverse the process described in the proof of \textbf{1} when we restrict to fillings using these numbers. 
So we take any SSYT $\mathcal{T}_2$ of shape $\mathcal{S}(\lambda^{t},k,n)$ with lattice reading word whose entries are from the set $\{1,2,\ldots, n+1\}$. 
Therefore, the foundation $\lambda^{t}$ consists of a strictly increasing sequence $a_1 < a_2 < \ldots < a_{\lambda_1}$ and $a_{\lambda_1}\leq n+1$. 

As before, we can check that the tableau of shape $\lambda \oplus \Delta_n$ has lattice reading word. 
If not, then since the content of $\Delta_n$ was $(n,n-1,\ldots, 3,2,1)$ and $a_i \leq n+1$ for each $i$, we would require two of the $a_i$ to be equal. Since the $a_i$ are distinct, this tableau does have a lattice reading word. 
We may now use Theorem~\ref{fatjoinlemma} to obtain a SSYT $\mathcal{T}_1$ of shape $s_{\mathcal{S}(\lambda, k,n)}$.

Also, we can recover the tableau $\mathcal{T}_2$ from $\mathcal{T}_1$ by transposing the entries of the foundation $\lambda$.
This shows that $s_{\mathcal{S}(\lambda, k,n)}-s_{\mathcal{S}(\lambda^{t},k,n)} \geq_s 0$ as functions in $n+1$ variables. This and \textbf{1} completes the proof that $s_{\mathcal{S}(\lambda^{t}, k,n)} =_{n+1} s_{\mathcal{S}(\lambda,k,n)}$. 

\

\noindent \textbf{3.}

If $\lambda \neq \lambda^{t}$ then $\lambda_1 \geq 2$. 
Therefore $\lambda^{t}$ consists of a column of length $\lambda_1$, where $2 \leq \lambda_1 \leq n+k$. 
We shall fill the last two spaces of the column with $n+1$ and $n+2$, and then fill the remaining $\lambda_1-2$ spaces with $2-k, 3-k, \ldots, \lambda_1-1-k$. 
That is, we fill the column with the entries $2-k,3-k,\ldots, \lambda_1-1-k, n+1, n+2$. 
This filling gives rise to a SSYT of shape $\mathcal{S}(\lambda^{t}, k,n)$ with a lattice reading word. 
Further, no SSYT of shape $\mathcal{S}(\lambda, k,n)$ with the same content can have a lattice reading word since the lattice condition implies that the $n+1$ lies to the right of the $n+2$ and this violates the fact that the row weakly increases. 

\

This completes the proof of \textbf{1}, \textbf{2}, and \textbf{3}. 
We now show that the formula stated for $s_{\mathcal{S}(\lambda^{t}, k,n)} - s_{\mathcal{S}(\lambda,k,n)}$ is correct.
From \textbf{2} we know that the only terms $s_\nu$ with non-zero coefficient that appear in the difference have $l(\nu) \geq n+2$. 

Let $\nu$ be any partition with $l(\nu) \geq n+2$.
No SSYT of shape $\mathcal{S}(\lambda, k,n)$ with lattice reading word and content $\nu$ exists since the lattice condition requires that both the values $n+1$ and $n+2$ appear in $\lambda$, but the lattice condition also requires the $n+1$ appears to the right of the $n+2$ and this violates the fact that the row weakly increases.

Any SSYT of shape $\mathcal{S}(\lambda^{t}, k,n)$ with lattice reading word and content $\nu$ must contain both the values $n+1$ and $n+2$ in $\lambda^t$. 
These appear one above the other in $\lambda^t$. 
The lattice condition implies that any values below the entry $n+2$ in $\lambda^t$, assuming there are any, must be the values $n+3$, $n+4$, $\ldots$, until the end of the column is reached. 
The values of $\lambda^t$ above the entry $n+1$ form some set $A \subseteq \{ 1,2,\ldots,n \}$. 
In fact, for $k=0$ we must have $A \subseteq \{ 2,3,\ldots,n \}$ since there is an entry $1$ directly above $\lambda^t$.
Hence, in either case, $A \subseteq \{ 2-k,3-k,\ldots,n \}$.
The order that these values appear in $\lambda^t$ is uniquely determined since the column must strictly increase.
Since there are only $\lambda_1$ entries in $\lambda^t$ and we know that both $n+1$, $n+2$ $\in \lambda^t$, we have $|A| \leq \lambda_1 -2$.

Furthermore, any $A \subseteq \{ 2-k,3-k,\ldots,n \}$ with $|A| \leq \lambda_1 -2$ gives rise to a unique SSYT of shape $\mathcal{S}(\lambda^{t}, k,n)$ with lattice reading word that contains $A$ in its foundation. 
Namely, we fill the first $|A|$ boxes of $\lambda^t$ with the values of $A$ and the remaining $\lambda_1 - |A|$ boxes of $\lambda^t$ with $n+1,n+2,\ldots, n+\lambda_1 - |A|$.
The content of this filling is $\nu(A) = \delta_n + e_A +(0^n,1^{\lambda_1 -|A|})$.
Since $l(\nu(A)) \geq n+2$, there is no SSYT of shape $\mathcal{S}(\lambda, k,n)$ with lattice reading word and content $\nu(A)$.
Since all SSYT of shape $\mathcal{S}(\lambda^{t}, k,n)$ with lattice reading word and content $\nu$, for $l(\nu) \geq n+2$, arise in this manner, we have
\[s_{\mathcal{S}(\lambda^{t}, k,n)} - s_{\mathcal{S}(\lambda,k,n)} =   \displaystyle\sum_{ \stackrel{\mbox{\scriptsize $A \subset \{ 2-k,3-k,\ldots, n \}$}}{|A|\leq \lambda_1-2}}   s_{\nu(A)}, \]
as claimed. 

Now, if we have two sets $A$ and $A'$, with $\nu(A) = \nu(A')$, then looking at the first $n$ rows of these partitions gives that $\delta_n + e_A =\delta_n + e_{A'}$. Hence $e_A= e_{A'}$, so we obtain $A = A'$.
This shows that no distinct subsets contribute to the same term $s_\nu$. Therefore the difference is multiplicity-free. \qed

\end{proof}

\subsection{Two Part Partitions}
We now turn to the case of partitions with exactly two parts.
In Theorem~\ref{2theorem} the condition $\lambda_1>1$ is present to prevent $\lambda$ from being a single column. 
We wish to exclude this case since the previous section already inspected the difference when the foundations were a row and column, respectively. 

\begin{theorem}
\label{2theorem}
Let $n$ be a positive integer, $0 \leq k \leq 1$, and $\lambda$ is a partition with two rows such that $1<\lambda_1 \leq k+n$. Then 
\begin{enumerate}
\item $s_{\mathcal{S}(\lambda^{t}, k,n)}-s_{\mathcal{S}(\lambda,k,n)} \geq_s 0,$ 
\item $s_{\mathcal{S}(\lambda^{t}, k,n)}=_{n+1}s_{\mathcal{S}(\lambda,k,n)},$ and
\item $s_{\mathcal{S}(\lambda^{t}, k,n)} \neq_{n+2} s_{\mathcal{S}(\lambda,k,n)}$ if $\lambda \neq \lambda^{t}.$ 
\end{enumerate}
\end{theorem}

\begin{proof} 

\

\noindent \textbf{1.}

Consider any SSYT $\mathcal{T}_1$ of shape $\mathcal{S}(\lambda,k,n)$ with lattice reading word, where $0 \leq k \leq 1$. Then the entries of the foundation $\lambda$ have the following relations.

\begin{eqnarray*}
\  &  & a_1 <a_2 <\ldots <a_{\lambda_2} < \ldots < a_{\lambda_1}
\\ &  & \wedge \textrm{ } \textrm{ } \textrm{ } \textrm{ } \wedge \textrm{ } \textrm{ } \textrm{ } \textrm{ } \ldots \textrm{ } \textrm{ } \textrm{ } \wedge
\\ &  & b_1 \leq b_2 \textrm{ } \leq \ldots \leq b_{\lambda_2}
\\
\end{eqnarray*}

\noindent The fact that $a_1 <a_2 <\ldots < a_{\lambda_1}$ follows from Lemma~\ref{firstrowlem}, 
and the rest of the relations follow from the definition of a SSYT. 
Since each $a_i$ is positive, each $b_j \geq 2$. 
We also have the condition $a_1 > 1$ when $k=0$.

The relations $a_1 <a_2 <\ldots < a_{\lambda_1}$ show that there is no repeated value in the first row.
Further, if there is a repeated value $j$ in the second row, then $j \geq 2$ and the lattice condition shows that a $j-1$ occured in the first row, and that another $j$ could not have occured in the first row.
Therefore, if there is a repeated value in the foundation of $\mathcal{T}_1$, then it only appears twice. 
We shall create the desired tableau $\mathcal{T}_2$ of shape $\mathcal{S}( \lambda^{t}, k,n )$ in four steps that we shall describe presently.

In \textbf{Step 1} we will transpose the foundation of $\mathcal{T}_1$, giving us a tableau $T$ of shape $\lambda^t$, where $T$ is not necessarily semistandard. 
In \textbf{Step 2} we will check whether or not $n+2$ appears in our tableau $T$. If it does, we will swap its position so that it will be read after reading an $n+1$. 
This gives us a tableau $T'$ of shape $\lambda^t$ which still may not be semistandard, but will not violate the lattice condition when reading the value $n+2$. 
In \textbf{Step 3} we will fix any places in $T'$ where a column is not strictly increasing by rotating certain blocks of entries. 
After completing this step, we will have a SSYT $T''$ of shape $\lambda^t$.
In \textbf{Step 4} we will append $T''$ to $\Delta_n$, creating a SSYT $\mathcal{T}_2$ of shape $\mathcal{S}( \lambda^{t}, k,n )$.
During \textbf{Step 4}, we will have shown that $\mathcal{T}_2$ has a lattice reading word. 
Finally, we check $\mathcal{T}_1$ can be recovered from $\mathcal{T}_2$.

\

\noindent \textbf{Step 1:} \textit{Transpose Foundation} ($\mathcal{T}_1 \rightarrow T$)

\

Let us consider the tableau $T$ that is obtained by transposing the entries of the foundation of $\mathcal{T}_1$. Then $T$ is a tableau of shape $\lambda^{t}$.

\

\begin{center}
\begin{tabular}{ccc}

$a_1$ & $<$ & $b_1$ \\
$\wedge$ & & $\wedgeline$ \\
& $\vdots$ & \\
$\wedge$ & & $\wedgeline$ \\
$a_{\lambda_2-1}$ & $<$ & $b_{\lambda_2-1}$  \\
$\wedge$ & & $\wedgeline$ \\
$a_{\lambda_2}$ & $<$ & $b_{\lambda_2}$ \\
$\wedge$ & &  \\
$\vdots$ &&  \\
$\wedge$ & &  \\
$a_{\lambda_1}$ & & 
\end{tabular}
\end{center}

\

\noindent \textbf{Step 2:} \textit{Fix Potential $n+2$ Lattice Problems} ($T \rightarrow T'$)

\

The tableau $T$ cannot be immediately extended to a tableau of shape $\mathcal{S}(\lambda^{t}, k,n)$ with lattice reading word if there was both a $n+1$ and a $n+2$ in $T$. 
Since $\mathcal{T}_1$ had a lattice reading word, the tableau $T$ can have at most one $n+2$, which must occur at the bottom of the second column. 
Further, having an $n+2$ requires an $n+1$ in the first column of $T$, which must also appear at the bottom. 
We split into three cases and define a tableau $T'$ of shape $\lambda^{t}$ in each case. 
In each case we display the resulting tablueau $T'$.

\begin{enumerate}

\item If $a_{\lambda_1} \leq n+1$ and $b_{\lambda_2} < n+2$ then $T' =T$.

\item If $a_{\lambda_1} = n+1$ and $b_{\lambda_2} = n+2$, then
   \begin{enumerate}
   \item If $\lambda_1 > \lambda_2$ then $T'=T$ with $a_{\lambda_1}$ and $b_{\lambda_2}$ swapped.
   \item If $\lambda_1 = \lambda_2$ then $T'=T$ with $a_{\lambda_2}$ and $b_{\lambda_2 -1}$ swapped.
   \end{enumerate}
\end{enumerate}

\begin{tabular}{ccccccccccc}
&1.& &&&2.(a)&  &&&2.(b)&\\
&& && &&  &&&&\\
$a_1$ & $<$ & $b_1$ && $a_1$ & $<$ & $b_1$  && $a_1$ & $<$ & $b_1$\\
$\wedge$ & & \wedgeline && $\wedge$ & & \wedgeline  && $\wedge$ & & \wedgeline\\
& $\vdots$ & && & $\vdots$ &  && & $\vdots$ &\\
$\wedge$ & & \wedgeline && $\wedge$ & & \wedgeline  && $\wedge$ & & \wedgeline\\
$a_{\lambda_2-1}$ & $<$ & $b_{\lambda_2-1}$  && $a_{\lambda_2-1}$ & $<$ & $b_{\lambda_2-1}$  && $a_{\lambda_2-1}$ & $<$ & $n+1$\\
$\wedge$ & & \wedgeline && $\wedge$ & & \wedgeline  &&$\wedge$ & & $\wedge$\\
$a_{\lambda_2}$ & $<$ & $b_{\lambda_2}$ && $a_{\lambda_2}$ & $<$ & $n+1$  && $b_{\lambda_2-1}$ & $<$ & $n+2$\\
$\wedge$ & &  && $\wedge$ & &  &&&&\\
$\vdots$ &&  && $\vdots$ &&  &&&&\\
$\wedge$ & &  && $\wedge$ & &  &&&&\\
$a_{\lambda_1}$ & &  && $n+2$ & &
\end{tabular}

\

\noindent The lattice condition implies that there was only one $n+2$ in the second column and only one $n+1$ in the first column. 
Thus, the inequalities shown above are accurate.
Now, when reading any of these resulting tableaux $T'$, we shall always read an $n+1$ before we read the $n+2$.

\

\noindent \textbf{Step 3:} \textit{Fix Strictly Increasing Problems} ($T' \rightarrow T''$)

\

We first recall that each entry in the second column of $T'$ is at least 2. 
Further, if we had the entry 2 repeated in the second column of $T'$, then the semistandard condition on $\mathcal{T}_1$ would have required two of the $a_i$ equal to 1. 
However, we know that the $a_i$ are distinct, hence we cannot have the entry 2 being repeated in the second column of $T'$.

Let $j$ be such that $3 \leq j \leq n+1$ and suppose that the pair of entries {\tiny$\left[ \begin{array}{r}
          j\\
          j\\
       \end{array}\right]$} appear in the second column of $T'$. 
That is, one $j$ appears immediately below another $j$. 
Since $\mathcal{T}_1$ has a lattice reading word, a $j-1$ must have appeared in the first column of $T'$. 
Since the first row of $\mathcal{T}_1$ and hence the first column of $T'$ strictly increases there can only be one $j-1$ in this column. 
We now consider which positions in this column of $T'$ could this entry $j-1$ appear. 
There are two cases to consider.

\begin{enumerate}
\item The $j-1$ appears directly to the left of the bottom $j$.
\item The $j-1$ appears somewhere below the box directly to the left of the bottom $j$.
\end{enumerate}

\noindent The $j-1$ cannot occur above the left of the bottom $j$ else the element $x$ directly to the left of the bottom $j$ would satisfy $j-1<x<j$.

In this step we shall work from the top to the bottom of the second column of $T'$. 
We begin with $j$ being the first pair of repeated elements in the second column, and proceed downwards. 
We now show how we fix these strictly increasing problems in each of the two cases mentioned above.

\

\

\noindent \textbf{Case I:} \textit{The $j-1$ appears directly to the left of the bottom $j$.}

\

Starting from the top $j$, we search up the second column. 
Suppose the entry above $j$ is the value $j-1$. 
We have already mentioned that there is another $j-1$ in the first column of $T'$. 
Since there are two $j-1$'s in $T'$, there are also two $j-1$'s in $\lambda$, and since $\mathcal{T}_1$ had a lattice reading word, there must be a $j-2$ in $\lambda$. 
Hence, there is a $j-2$ in $T'$ as well. 
If this $j-2$ appeared in the second column of $T'$, then it must appear directly above the $j-1$. 
If this is the case, then in $\mathcal{T}_1$ we wouldn't read this $j-2$ until after we had read both $j-1$'s. 
Hence if this $j-2$ is in the second column then there is also a $j-2$ in the first column, and it must appear directly above the $j-1$ in the first column. 

There are now two $j-2$'s in $T'$. 
Using the lattice condition on $\mathcal{T}_1$, this implies there is at least one $j-3$ in $T'$. 
Just as with the $j-2$ previously mentioned, if this $j-3$ appears in the second column of $T'$ then we will find that there is a $j-3$ in the first column of $T'$, which appears directly above the $j-2$. 
Continuing in this manner, we have a sequence $j-1,j-2,\ldots, j-m$ above the top $j$ in the second column and corresponding sequence $j-1,j-2,\ldots, j-m$ in the first column, and the lattice condition of $\mathcal{T}_1$ requires a $j-m-1$ to appear in $T'$. 

Since the tableau is finite, the above procedure must terminate for some value $m=i$ and we find that the element $j-i-1$ \textit{does not} occur in the second column. Thus the $j-i-1$ appears in the first column, directly above the $j-i$. 
The entries $j-1,j-2,\ldots, j-i,j-i-1$ in the first column and the entries $j,j,j-1,j-2, \ldots, j-i$ in the second column define a \textit{block} of entries.
We highlight the block in the diagram below. 
Consider rotating this block of entries clockwise as shown.

\

\begin{tabular}{ccc ccc ccc}
$*$      & $\leq$   &$y$       & &                 &  $*$      & $\leq$   &$y$       \\
$\wedge$ &          &$\wedge$  & &                 &  $\wedge$ &          &$\wedge$  \\
$x$      &$<$       &\boldmath{$j-i$}   & &                 &  $x$      &$<$       &\boldmath{$j-i-1$}     \\
$\wedge$ &          &$\wedge$  & &                 &  $\wedge$ &          &$\wedge$  \\
\boldmath{$j-i-1$}  &$<$       &\boldmath{$j-i+1$}   & &                 &  \boldmath{$j-i$}    &$=$       &\boldmath{$j-i$}     \\
$\wedge$ &          &$\wedge$  & &                 &  $\wedge$ &          &$\wedge$  \\
\boldmath{$j-i$}    &$<$       &\boldmath{$j-i+2$}   & &$\longrightarrow$&  \boldmath{$j-i+1$}  &$=$      &\boldmath{$j-i+1$}   \\
$\wedge$ &          &$\wedge$  & &                 &  $\wedge$ &          &$\wedge$  \\
$\vdots$ & $\vdots$ &$\vdots$  & &                 &  $\vdots$ & $\vdots$ &$\vdots$  \\
$\wedge$ &          &$\wedge$  & &                 &  $\wedge$ &          &$\wedge$  \\
\boldmath{$j-3$}    &$<$       &\boldmath{$j-1$}     & &                 &  \boldmath{$j-2$}    &$=$       &\boldmath{$j-2$}     \\
$\wedge$ &          &$\wedge$  & &                 &  $\wedge$ &          &$\wedge$  \\
\boldmath{$j-2$}    &$<$       &\boldmath{$j$}       & &                 &  \boldmath{$j-1$}    &$=$       &\boldmath{$j-1$}       \\
$\wedge$ &          &$\|$      & &                 &  $\wedge$ &          &$\wedge$      \\
\boldmath{$j-1$}    &$<$       &\boldmath{$j$}       & &                 &  \boldmath{$j$}      &$=$       &\boldmath{$j$}       \\

$\wedge$ &          &$\wedge$      & &                 &  $\wedge$ &          &$\wedge$      \\
$z$      &$<$    &$w$       & &                 &  $z$      &$<$       &$w$       

\end{tabular}

\

This \textit{block rotation} fixes the strictly increasing problem caused by the $j$'s in the second column of $T'$.
We make a few comments to justify the accuracy of the relations displayed in the above diagrams. 

First, a note on the relation $* \leq y$ that appears at the top of each diagram. We are fixing these {\tiny $\left[ \begin{array}{r}
          j\\
          j\\
       \end{array}\right]$} problems in $T'$ from top to bottom. Initially all rows of $T'$ strictly increase since the columns of $T$ were strictly increasing. 
Yet, after we perform one or more of these block rotations all rows contained in a given block, except the top row, will now display equality. 
However, we will soon show that any two blocks are disjoint (see p.~\pageref{nonoverlap}), hence all other row equalities appear strictly above the current block that is displayed.

Initially, the entries of the first column of $T'$ were strictly increasing. 
This gives that $x < j-i-1$ and $j-1 < z$. 
The first inequality gives $x< j-i$. 
Also, there are can only be two $j$'s in the diagram, we have that $z \neq j$. 
Therefore the second inequality gives that $j < z$.
This shows that the block rotation preserves the strictly increasing nature of the first column.

Since we are working from top to bottom, we know that the second column strictly increases above the $j$'s. 
This gives $y < j-i$ and $j<w$. 
Since we assumed that $j-i-1$ did not appear in the second column, the first inequality gives $y < j-i-1$.
Therefore the strictly increasing segment of the second column has been extended downwards.

Again, since we are working from top to bottom and the blocks are disjoint, the rows strictly increase for each row that the current block intersects. 
After rotating this block it is evident that we obtain row equalities for each of these rows except the top-most.

Thus we have seen how to resolve repeated elements in the second column that falls into the first case. 
We now turn to the second case.

\

\

\noindent \textbf{Case II:} \textit{The $j-1$ appears somewhere below the box immediately left of the bottom $j$.}

\

We define the block of entries exactly as in \textbf{Case I}. 
The only difference in this case is that the block of entries consists of a sub-block in the second column and a lower sub-block in the first column. 
These two sub-blocks can be row-disjoint from one another. 
We still must check that we obtain a SSYT after we rotate the block. 
The strict inequality down the columns follows for the same reasons as in the first case. 
The rows were strictly increasing to begin with since the columns of $\mathcal{T}_1$ strictly increase. 
We need to check that, after rotating, the rows still weakly increase. 
There are three cases to check, each of which is displayed in the following diagram.

\begin{enumerate}
\item The row of interest $r_1$ is above the sub-block in the first column.
\item The row of interest $r_2$ is within the sub-block in both columns.
\item The row of interest $r_3$ is below the sub-block in the second column.
\end{enumerate}

\

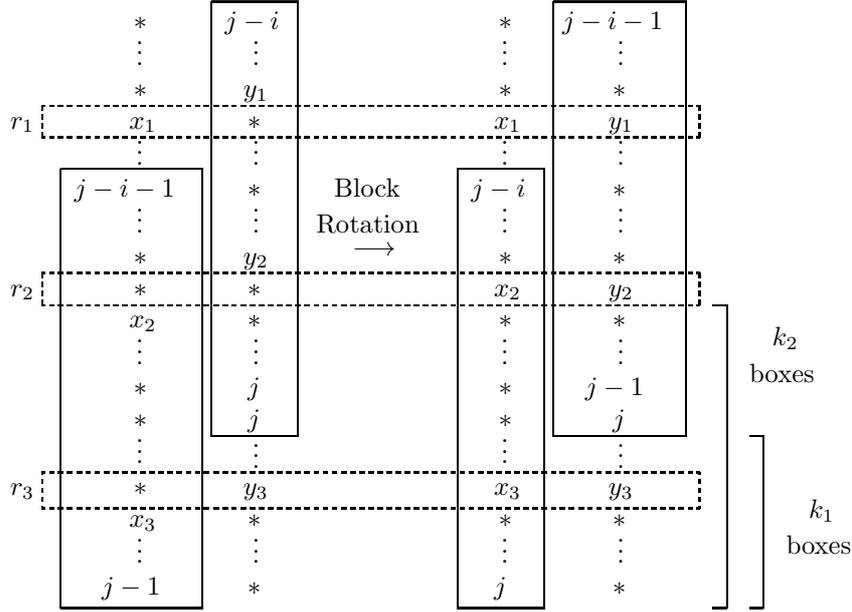
\setlength{\unitlength}{0.6mm}
\begin{picture}(100,100)(-10,0)

\put(20,100){$*$} \put(40,100){$j-i$}
\put(20.5,92){$\vdots$} \put(46,92){$\vdots$}
\put(20,85){$*$} \put(44,85){$y_1$}
\put(19,78){$x_1$} \put(45,78){$*$}
\put(20.5,70){$\vdots$} \put(46,70){$\vdots$}
\put(7,63){$j-i-1$} \put(45,63){$*$}
\put(20.5,55){$\vdots$} \put(46,55){$\vdots$}
\put(20,48){$*$} \put(44,48){$y_2$}
\put(20,41){$*$} \put(45,41){$*$}
\put(19,34){$x_2$} \put(45,34){$*$}
\put(20.5,26){$\vdots$} \put(46,26){$\vdots$}
\put(20,19){$*$} \put(45,19){$j$}
\put(20,12){$*$} \put(45,12){$j$}
\put(20.5,4){$\vdots$} \put(46,3){$\vdots$}
\put(20,-3){$*$} \put(44,-3){$y_3$}
\put(19,-10){$x_3$} \put(45,-10){$*$}
\put(20.5,-18){$\vdots$} \put(46,-18){$\vdots$}
\put(13,-25){$j-1$} \put(45,-25){$*$}

\put(4,-28){\line(0,1){97}}
\put(35,-28){\line(0,1){97}}
\put(4,-28){\line(1,0){31}}
\put(4,69){\line(1,0){31}}

\put(37,10){\line(0,1){96}}
\put(56,10){\line(0,1){96}}
\put(37,106){\line(1,0){19}}
\put(37,10){\line(1,0){19}}

\put(64,63){\textrm{Block}}
\put(60,55){\textrm{Rotation}}
\put(68,50){$\longrightarrow$}

\put(100,100){$*$} \put(114,100){$j-i-1$}
\put(100.5,92){$\vdots$} \put(126,92){$\vdots$}
\put(100,85){$*$} \put(125,85){$*$}
\put(99,78){$x_1$} \put(124,78){$y_1$}
\put(100.5,70){$\vdots$} \put(126,70){$\vdots$}
\put(94,63){$j-i$} \put(125,63){$*$}
\put(100.5,55){$\vdots$} \put(126,55){$\vdots$}
\put(100,48){$*$} \put(125,48){$*$}
\put(99,41){$x_2$} \put(124,41){$y_2$}
\put(100,34){$*$} \put(125,34){$*$}
\put(100.5,26){$\vdots$} \put(126,26){$\vdots$}
\put(100,19){$*$} \put(119,19){$j-1$}
\put(100,12){$*$} \put(125,12){$j$}
\put(100.5,4){$\vdots$} \put(126,3){$\vdots$}
\put(99,-3){$x_3$} \put(124,-3){$y_3$}
\put(100,-10){$*$} \put(125,-10){$*$}
\put(100.5,-18){$\vdots$} \put(126,-18){$\vdots$}
\put(99,-25){$j$} \put(125,-25){$*$}

\put(91,-28){\line(0,1){97}}
\put(110,-28){\line(0,1){97}}
\put(91,-28){\line(1,0){19}}
\put(91,69){\line(1,0){19}}

\put(112,10){\line(0,1){96}}
\put(141,10){\line(0,1){96}}
\put(112,106){\line(1,0){29}}
\put(112,10){\line(1,0){29}}

\put(0,76){\dashbox{1}(144,7)[tl]{ }}

\put(0,39){\dashbox{1}(144,7)[tl]{ }}

\put(0,-6){\dashbox{1}(144,8)[tl]{ }}

\put(-7,78){$r_1$}
\put(-7,41){$r_2$}
\put(-7,-3){$r_3$}

\put(158,-28){\line(0,1){38}}
\put(158,-28){\line(-1,0){3}}
\put(158,10){\line(-1,0){3}}
\put(168,-8){$k_1$}
\put(163,-16){\textrm{boxes}}

\put(150,-28){\line(0,1){67}}
\put(150,-28){\line(-1,0){3}}
\put(150,39){\line(-1,0){3}}
\put(160,30){$k_2$}
\put(155,22){\textrm{boxes}}

\end{picture}

\

\

\

\

\

\noindent Thus, in the left-hand diagram, the value $x_3$ may in fact be the bottom-most entry of the sub-block in the first column. That is, $x_3=j-1$ is a case considered by the various possible positions of $r_3$.

For the rows $r_1$, $r_2$, and $r_3$ shown, we wish to show $x_i \leq y_i$ for $i=1,2,3$. 
This will show that the rows weakly increase after the block rotation. 
Since the columns strictly increase we have
\[ x_1 < j-i-1 < j-i \leq y_1, \textrm{ and} \]
\[ x_3 \leq j-1 < j < y_3. \]
\noindent Further, with $k_1$ and $k_2$ as shown, we have $x_2=j-k_2$ and $y_2 = j-k_2+k_1$, hence
\[ x_2 \leq x_2 + k_1 =j - k_2 + k_1 = y_2.\]
\noindent Therefore the rows of the resulting tableau weakly increase.

\

We can now check that no two block rotations move the same elements. \label{nonoverlap} 
If we have performed a block rotation to fix a {\tiny $\left[ \begin{array}{r}
          j\\
          j\\
       \end{array}\right]$} problem then there cannot be two $j+1$'s below the new position of the $j$'s, since, if there were, then in $\mathcal{T}_1$ we would have read two $j+1$'s in the foundation $\lambda$ before reading any $j$'s, which would violate the lattice condition. 
Since the blocks are formed by extending consecutively increasing sequences, this shows that the next block could not overlap this block.
Thus any two blocks are disjoint.

This process of block rotation removes the pairs of repeating elements in the second column while maintaining the strict inequalities of the two columns up to that point and never violates the weakly increasing restrictions on the rows. 
By working top to bottom and repeating this process for each pair of repeated elements in the second column we shall produce a SSYT $T''$ of shape $\lambda^{t}$.

\newpage

\noindent \textbf{Step 4:} \textit{Extend from shape $\lambda^{t}$ to a SSYT of shape $\mathcal{S}(\lambda^{t},k,n)$} ($T'' \rightarrow \mathcal{T}_2$)

\

Using $T''$ and the unique semistandard filling of $\Delta_n$ with lattice reading word, we obtain a SSYT $T_2$ of shape $\lambda^t \oplus \Delta_n$.

We also claim that ${T}_2$ has a lattice reading word. 
In each case there is at most one $n+2$ in ${T}_2$. 
If an $n+2$ does appear, then it appeared at the end of the second column of $T$, and none of the block rotations in \textbf{Step 3} moved its position. 
Further, since $\mathcal{T}_1$ was lattice, we know that an $n+1$ must have appeared at the end of the first column of $T$. 
If $\lambda_1>\lambda_2$ then Case 2(a) placed the $n+2$ at the bottom of the first column and placed the $n+1$ at the bottom of the second column. 
The only way this $n+1$ could have moved in \textbf{Step 3} was if there was another $n+1$ above it. 
In either case, an $n+1$ remains at the end of the second column, so we read this $n+1$ before reading the $n+2$.
Similarly, if $\lambda_1=\lambda_2$ then Case 2(b) placed the $n+2$ at the bottom of the second column and placed the $n+1$ directly above it. 
Again, the only way this $n+1$ could have moved in \textbf{Step 3} was if there was another $n+1$ above it, 
and in either case, an $n+1$ remains above the $n+2$, so we read this $n+1$ before the $n+2$.
Therefore the single $n+2$ that occurs does not violate the lattice condition.

Nevertheless, suppose that the reading word of ${T}_2$ is not lattice. 
Let $j+1$ be the entry that, when read, violated the lattice condition. 
We know that the content after reading $\Delta_n$ is $(n,n-1,\ldots,3,2,1)$, 
so we must have read at least two more $j+1$'s than $j$'s since we finished reading $\Delta_n$. 
Since there are only two column for these two $j+1$'s to appear in, there can only be two $j+1$'s. 
Thus we must have read exactly two $j+1$'s and no $j$'s since we finished reading $\Delta_n$. 
Let the values of the columns be $x_1,x_2,\ldots x_{\lambda_1}$ and $y_1,y_2,\ldots,y_{\lambda_2}$ respectively. Thus $x_i = j+1$ and $y_k = j+1$ for some $i \geq k$.

We know $j \not\in \{x_{i}, x_{i+1}, \ldots, x_{\lambda_1} \} \cup \{y_k, y_{k+1}, \ldots, y_{\lambda_2} \}$ since the columns strictly increase. But we also know $j \not\in \{x_1,x_2,\ldots,x_i \} \cup \{y_1,y_2,\ldots,y_i\}$ since by assumption we had not read any $j$'s when reaching the second $j+1$. Thus 
\[j \not\in \{x_1,x_2,\ldots,x_i \} \cup \{y_1,y_2,\ldots,y_i\} \cup \{x_{i}, x_{i+1}, \ldots, x_{\lambda_1} \} \cup \{y_k, y_{k+1}, \ldots, y_{\lambda_2} \}, \]
which is to say $j \not\in \{x_1,x_2,\ldots x_{\lambda_1}, y_1,y_2, \ldots, y_{\lambda_2}\}$. This is a contradiction since $\{x_1,x_2,\ldots x_{\lambda_1}, y_1,y_2, \ldots, y_{\lambda_2}\} = \{a_1,a_2,\ldots a_{\lambda_1}, b_1,b_2, \ldots, b_{\lambda_2}\}$ and the fact that there are two $j+1$'s in the foundation of $\mathcal{T}_1$ implies there was at least one $j$ as well, else $\mathcal{T}_1$ could not have a lattice reading word.
Therefore the SSYT ${T}_2$ has a lattice reading word.

Thus we may apply Lemma~\ref{fatjoinlemma} to obtain a SSYT $\mathcal{T}_2$ of shape $\mathcal{S}(\lambda^t,k,n)$ with lattice reading word.

\

\

\noindent \textbf{Check:} \textit{Can We Recover $\mathcal{T}_1$?}

\

We claim that $\mathcal{T}_1$ can be recovered from $\mathcal{T}_2$. 
\textbf{Step 4} can be undone by passing back to the foundation of $\mathcal{T}_2$. 
The swapping performed in \textbf{Step 2} is easily detected and reversed. 
Undoing \textbf{Step 1} requires a mere transposition of diagrams. 
So to prove the claim it remains to check that the block rotations in \textbf{Step 3} are reversible. 
Since we have noted that no two block rotations move a common element, it is enough to check that each of the two cases are reversible. 
Hence we look at reversing both cases after fixing a {\tiny $\left[ \begin{array}{r}
          j\\
          j\\
       \end{array}\right]$} problem.

If we made such a manipulation in \textbf{Case I} of \textbf{Step 3} (i.e. $j-1$ \textit{was} immediately left of the bottom $j$), then the foundation of $\mathcal{T}_2$ \textit{now} has a row containing {\tiny $\left[ \begin{array}{rr}
          j & j\\
          \end{array}\right]$}. 
By transposing the foundation we obtain a tableau of shape $\lambda$ with a column {\tiny $\left[ \begin{array}{r}
          j\\
          j\\
       \end{array}\right]$}. 
We now search to the left of this column until we come upon a nonrepeating column {\tiny $\left[ \begin{array}{c}
          x\\
          j-i-1\\
       \end{array}\right]$} where $x<j-i-1$. This determines the block of entries we wish to rotate. 
We then rotate the entries clockwise to undo the original block rotation. 

\

\

\noindent \begin{tabular}{ccc ccc ccc ccc ccc }
  $*$      & $<$ &  $x$      & $<$ & \boldmath{$j-i$} & $<$ &  $\ldots$ & $<$ & \boldmath{$j-1$} & $<$ & \boldmath{$j$}  &$<$&$*$      \\
  $\wedge$ &     &  $\wedge$ &     & $\|$  &     &  $\ldots$ &     &  $\|$ &     & $\|$ &   &$\wedge$        \\
  $y$      & $<$ &  \boldmath{$j-i-1$}  & $<$ & \boldmath{$j-i$} & $<$ &  $\ldots$ & $<$ & \boldmath{$j-1$} & $<$ & \boldmath{$j$}  &$<$&$*$      \\
&&&&&&&&&&&& \\
&&&&&$\downarrow$&&&&&&& \\
&&&&&&&&&&&& \\
  $*$      & $<$ &  $x$      & $<$ & \boldmath{$j-i-1$} & $<$ &  $\ldots$ & $<$ & \boldmath{$j-2$} & $<$ & \boldmath{$j-1$} &$<$&$*$      \\
  $\wedge$ &     &  $\wedge$ &     & $\wedge$  &     &  $\ldots$ &     &  $\wedge$ &     & $\wedge$ &   &$\wedge$   \\
  $y$      & $<$ &  \boldmath{$j-i$}  & $<$ & \boldmath{$j-i+1$} & $<$ &  $\ldots$ & $<$ & \boldmath{$j$} & $=$ & \boldmath{$j$}    &$<$&$*$

\end{tabular}

\

\

If we made such a manipulation in the second case (i.e. $j-1$ was below the left of the bottom $j$), then upon transposing the foundation of $\mathcal{T}_2$ and reading its entries we come across two $j-i$'s before reading the $j-i-1$. This determines the smallest element $j-i-1$ in the block we wish to form. 
The $j-i-1$ necessarily appears in the second row, immediately to the left of the second $j-i$ that we have just read. The remaining elements of the block are the sequence of consecutive integers $j-i,j-i+1,\ldots,j$ to its right and the copy of this sequence in the first row.
The consecutive sequence cannot be further extended to the right, otherwise, in the foundation of $\mathcal{T}_1$, we would have read two $j+1$'s before reading either of the $j$'s. 
Therefore the correct endpoint $j$ for this block of entries can be correctly determined. 
Again we rotate the elements in the block clockwise to recover the original. 
Therefore we can undo the block rotations of \textbf{Step 3}. 
Hence $\mathcal{T}_1$ can be recovered from $\mathcal{T}_2$. 
This completes the proof that $s_{\mathcal{S}(\lambda^{t}, k,n)}-s_{\mathcal{S}(\lambda,k,n)} \geq_s 0$.

An example of block rotations is shown after the remainder of the proof.

\

\

\noindent \textbf{2.}

We now show that we can reverse the process described in the proof of \textbf{1} when we restrict to fillings using the numbers $\{1,2,\ldots, n+1\}.$
If we take any SSYT $\mathcal{T}_2$ of shape $\mathcal{S}(\lambda^{t},k,n)$ with lattice reading word then the foundation $\lambda^{t}$ consists of two strictly increasing columns. 

\begin{center}
\begin{tabular}{ccc}

$a_1$ & $<$ & $b_1$ \\
$\wedge$ & & $\wedge$ \\
$a_2$ & $\leq$ & $b_2$ \\
$\wedge$ & & $\wedge$ \\
& $\vdots$ & \\
$\wedge$ & & $\wedge$ \\
$a_{\lambda_2-1}$ & $\leq$ & $b_{\lambda_2-1}$  \\
$\wedge$ & & $\wedge$\\
$a_{\lambda_2}$ & $\leq$ & $b_{\lambda_2}$ \\
$\wedge$ & &  \\
$\vdots$ &&  \\
$\wedge$ & &  \\
$a_{\lambda_1}$ & & 
\end{tabular}
\end{center}

Since the columns strictly increase we may see a pair of equal elements $a_i=b_i$ in $\lambda^{t}$, but never three equal values. 
We have $a_1 < b_1$ by Lemma~\ref{firstrowlem}, which implies that each $b_j \geq 2$. 
We shall create the desired tableau $\mathcal{T}_1$ of shape $\mathcal{S}( \lambda, k,n )$ using three steps that we shall describe presently.

In \textbf{Step A} we will transpose the foundation of $\mathcal{T}_2$, giving us a tableau $T$ of shape $\lambda$, which is not necessarily semistandard. 
In \textbf{Step B} we will fix any places in $T$ where the lattice condition would fail if we appended $T$ to $\Delta_n$ by rotating certain blocks of entries. 
After completing this step, we will have a SSYT $T'$ of shape $\lambda$.
In \textbf{Step C} we will append $T'$ to $\Delta_n$, creating a SSYT $\mathcal{T}_1$ of shape $\mathcal{S}( \lambda, k,n )$.
During \textbf{Step C}, we will have shown that $\mathcal{T}_1$ has a lattice reading word. 
Finally, we show that $\mathcal{T}_2$ can be recovered from $\mathcal{T}_1$.

\

\noindent \textbf{Step A:} \textit{Transpose Foundation}  ($\mathcal{T}_2 \rightarrow T$)

Let us consider the tableau $T$ that is obtained by transposing the entries of the foundation of $\mathcal{T}_2$. Then $T$ is a tableau of shape $\lambda$.
\begin{eqnarray*}
\  &  & a_1 <a_2 <\ldots <a_{\lambda_2} < \ldots < a_{\lambda_1}
\\ &  & \textrm{ } \wedgeline \textrm{ } \textrm{ } \textrm{ } \textrm{ } \textrm{ } \textrm{ } \textrm{ } \wedgeline \textrm{ } \textrm{ } \textrm{ } \textrm{ }  \textrm{ } \ldots \textrm{ } \textrm{ } \textrm{ } \textrm{ } \textrm{ } \wedgeline
\\ &  & b_1 < b_2 \textrm{ } < \ldots < b_{\lambda_2}
\\
\end{eqnarray*}

\noindent \textbf{Step B:} \textit{Fix Future Lattice Problems} ($T \rightarrow T'$)

In this step we move certain entries of $T$ so that the resulting tableau of shape $\lambda$ can be extended to a tableau $\mathcal{T}_1$ of shape $\mathcal{S}(\lambda, k,n)$ whose reading word is lattice. 
Given $T$, the only problems that may arise is that, when reading the entries of $T$, we come to a point where we have read two $j$'s before reading any $j-1$'s. We call this a \textit{future lattice problem}. We look at how to remedy these problems.

Let $j$ be the first value that we have read twice without reading any $j-1$. 
In this case we say that we have a $j,j$ problem. 
Since $\mathcal{T}_2$ has a lattice reading word, a $j-1$ must have appeared somewhere in $T$. 
The only place we have not read is to the left of the $j$ in the second row. 
Since the rows of $T$ weakly increase (in fact strictly increase, to begin with), the $j-1$ must appear immediately to the left of the $j$ in the bottom row. 
We note that we must have $j \geq 3$, since 1 cannot appear in the bottom row. 
We now search to the right of both $j$'s and find the largest consecutive sequence $j,j+1,\ldots,j+i$ that appears in both rows. 
These sequences together with the entry $j-1$ define a \textit{block}.
There are two cases to consider:

\begin{enumerate}
\item The two $j$'s are in the same column.
\item The top $j$ appears to the right of the bottom $j$.
\end{enumerate}
\noindent We note that the top $j$ cannot occur to the left of the bottom $j$, otherwise the entry $m$ above the bottom $j$ satisfies $j < m \leq j$.
The two cases are displayed below.

\setlength{\unitlength}{0.55mm}
\begin{picture}(0,55)(20,60)

\put(36,100){$x$} \put(52,100){$j$}       \put(62,100){$j+1$}  \put(79,100){$\cdots$} \put(118,100){$\cdots$} \put(129,100){$j+i-1$} \put(160,100){$j+i$} \put(182,100){$z$}

\put(18,85){$y$} \put(30,85){$j-1$}        \put(52,85){$j$}  \put(62,85){$j+1$} \put(79,85){$\cdots$} \put(118,85){$\cdots$}  \put(129,85){$j+i-1$} \put(160,85){$j+i$} \put(182,85){$w$}

\put(115,70){or}

\put(73,55){$x$} \put(87,55){$j$}       \put(97,55){$j+1$}  \put(115,55){$\cdots$} \put(156,55){$\cdots$} \put(169,55){$j+i-1$} \put(200,55){$j+i$} \put(226,55){$z$}

\put(18,40){$y$} \put(30,40){$j-1$}        \put(52,40){$j$}  \put(62,40){$j+1$} \put(79,40){$\cdots$} \put(118,40){$\cdots$}  \put(129,40){$j+i-1$} \put(160,40){$j+i$} \put(182,40){$w$}

\put(26,82){\framebox(152,10)[tl]{ }}
\put(48,97){\framebox(130,10)[tl]{ }}

\put(81,52){\framebox(140,10)[tl]{ }}

\put(27,37){\framebox(151,10)[tl]{ }}

\end{picture}

\

\

\

\

After rotating the entries of either block clockwise the columns will strictly increase. 
We inspect the relations of the rows in both of the two cases simultaneously. 
We may do this since the entries of the rows are the same in each case. 
We display the rotation in the second case. 
The entries of the block are rotated one position clockwise as shown.

\setlength{\unitlength}{0.55mm}
\begin{picture}(0,55)(20,60)

\put(73,100){$x$} \put(87,100){$j$}       \put(97,100){$j+1$}  \put(115,100){$\cdots$} \put(156,100){$\cdots$} \put(169,100){$j+i-1$} \put(200,100){$j+i$} \put(226,100){$z$}

\put(18,85){$y$} \put(30,85){$j-1$}        \put(52,85){$j$}  \put(62,85){$j+1$} \put(79,85){$\cdots$} \put(118,85){$\cdots$}  \put(129,85){$j+i-1$} \put(160,85){$j+i$} \put(182,85){$w$}

\put(115,70){$\downarrow$}

\put(73,55){$x$} \put(84,55){$j-1$}       \put(105,55){$j$}  \put(114,55){$\cdots$} \put(152,55){$\cdots$} \put(163,55){$j+i-2$} \put(194,55){$j+i-1$} \put(226,55){$z$}

\put(19,40){$y$} \put(32,40){$j$}        \put(41,40){$j+1$}  \put(61,40){$j+2$} \put(79,40){$\cdots$} \put(123,40){$\cdots$}  \put(135,40){$j+i$} \put(158,40){$j+i$} \put(182,40){$w$}

\put(26,82){\framebox(152,10)[tl]{ }}
\put(82,97){\framebox(139,10)[tl]{ }}

\put(81,52){\framebox(140,10)[tl]{ }}

\put(27,37){\framebox(151,10)[tl]{ }}

\end{picture}

\

\

\

\

Since both rows strictly increase to begin with, we have $x < j$, $y < j-1$, $j+i< z$, and $j+i< w$.
Therefore we have the inequalities $y < j$, $j+i-1 < z$ and $j+i < w$. 
Further, since there was no $j-1$ in the first row, we have $x<j-1$. 
Hence the only pair of equal values in a row that is created by this block rotation is the pair $j+i,j+i$ in the bottom-right of the block.

Initially, the columns of $T$ were weakly increasing. 
Now consider any column of $T$ with a pair of equal values, say {\tiny $\left[ \begin{array}{c}
                                                          m \\
                                                          m \\
                                                         \end{array} \right]$}. 
Since the reading word for $\mathcal{T}_2$ is lattice, there is a $m-1$ in $T$. 
This $m-1$ must occur in the column immediately to the left of the {\tiny $\left[ \begin{array}{c}
                                                          m \\
                                                          m \\
                                                         \end{array} \right]$}. 
If this $m-1$ appears in the first row of $T$, then since the colums weakly increase and there are only two $m$'s in $T$ we find that a second $m-1$ appears immediately below it.
Thus we obtain a column {\tiny $\left[ \begin{array}{c}
                                                          m-1 \\
                                                          m-1 \\
                                                         \end{array} \right]$}.
We continue extending this block of equal valued columns to the left.
It terminates when we reach some column {\tiny $\left[ \begin{array}{c}
                                                          j \\
                                                          j \\
                                                         \end{array} \right]$} 
where a $j-1$ appears immediately left of this column in the second row, but not in the first row.
\
\begin{center}
\noindent \begin{tabular}{ccccc ccccc ccc}
    $x$               & $<$     & \boldmath{$j$}   & $<$ &  $\ldots$ & $<$ & \boldmath{$m-1$} & $<$ & \boldmath{$m$}   \\ 
     $\wedge$         &         & $\|$             &     &  $\ldots$ &     & $\|$           &     & $\|$               \\ 
    \boldmath{$j-1$}  & $<$     & \boldmath{$j$}   & $<$ &  $\ldots$ & $<$ & \boldmath{$m-1$} & $<$ & \boldmath{$m$}   \\ 
\end{tabular}
\end{center}
\

\noindent This gives rise to a future lattice problem. 
Thus all columns in $T$ that did not strictly increase are contained in some block. 
We now check that the columns of each block strictly increase after performing the block rotation. There are three cases to consider.

\begin{enumerate}
\item The column of interest $c_1$ is left of the sub-block in the first row.
\item The column of interest $c_2$ is within the sub-block in both rows.
\item The column of interest $c_3$ is right of the sub-block in the second row.
\end{enumerate}

\

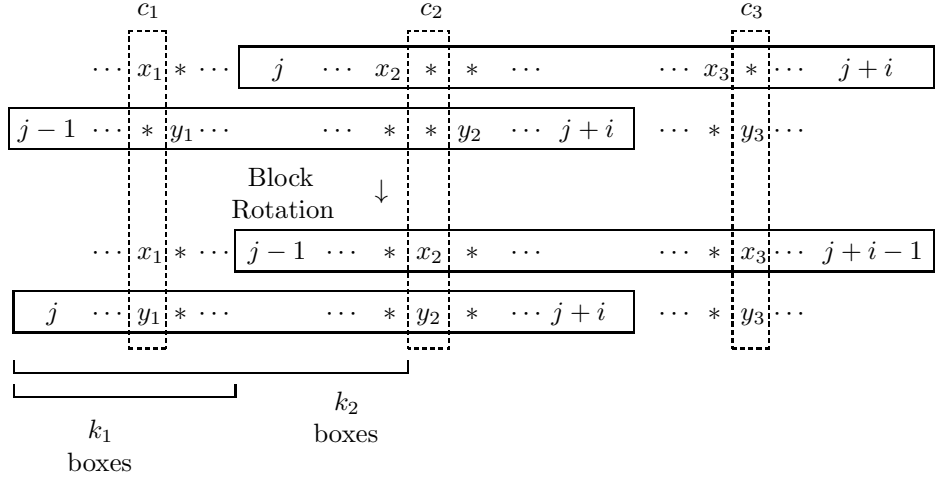
\setlength{\unitlength}{0.54mm}
\begin{picture}(0,55)(37,60)

\put(46,100){$\cdots$} \put(57,100){$x_1$} \put(66,100){$*$} \put(72,100){$\cdots$} \put(90,100){$j$} \put(102,100){$\cdots$} \put(115,100){$x_2$} \put(127,100){$*$} \put(137,100){$*$} \put(148,100){$\cdots$}    \put(184,100){$\cdots$}  \put(195,100){$x_3$}   \put(205,100){$*$}  \put(212,100){$\cdots$}      \put(228,100){$j+i$} 

\put(46,85){$\cdots$} \put(58,85){$*$} \put(65,85){$y_1$} \put(72,85){$\cdots$} \put(28,85){$j-1$}  \put(102,85){$\cdots$}  \put(117,85){$*$} \put(127,85){$*$} \put(135,85){$y_2$} \put(148,85){$\cdots$}    \put(160,85){$j+i$}  \put(184,85){$\cdots$}  \put(196,85){$*$}   \put(204,85){$y_3$}  \put(212,85){$\cdots$}

\put(115,70){$\downarrow$}

\put(46,55){$\cdots$} \put(57,55){$x_1$} \put(66,55){$*$} \put(72,55){$\cdots$}  \put(84,55){$j-1$}   \put(103,55){$\cdots$} \put(117,55){$*$} \put(125,55){$x_2$} \put(137,55){$*$} \put(148,55){$\cdots$}        \put(184,55){$\cdots$}  \put(196,55){$*$}   \put(204,55){$x_3$}  \put(212,55){$\cdots$}        \put(224,55){$j+i-1$} 

\put(46,40){$\cdots$} \put(57,40){$y_1$} \put(66,40){$*$} \put(72,40){$\cdots$}  \put(35,40){$j$}      \put(103,40){$\cdots$} \put(117,40){$*$} \put(125,40){$y_2$} \put(137,40){$*$} \put(148,40){$\cdots$}   \put(158,40){$j+i$}    \put(184,40){$\cdots$}  \put(196,40){$*$}   \put(204,40){$y_3$}  \put(212,40){$\cdots$}

\put(26,82){\framebox(152,10)[tl]{ }}
\put(82,97){\framebox(169,10)[tl]{ }}

\put(81,52){\framebox(170,10)[tl]{ }}

\put(27,37){\framebox(151,10)[tl]{ }}

\put(84,73){\textrm{Block}}
\put(80,65){\textrm{Rotation}}

\put(202,33){\dashbox{1}(9,78)[tl]{ }}

\put(123,33){\dashbox{1}(10,78)[tl]{ }}

\put(55,33){\dashbox{1}(9,78)[tl]{ }}

\put(57,115){$c_1$}
\put(126,115){$c_2$}
\put(204,115){$c_3$}

\put(27,21){\line(1,0){54}}
\put(27,21){\line(0,1){3}}
\put(81,21){\line(0,1){3}}

\put(45,11){$k_1$}
\put(40,3){\textrm{boxes}}

\put(27,27){\line(1,0){96}}
\put(27,27){\line(0,1){3}}
\put(123,27){\line(0,1){3}}

\put(105,18){$k_2$}

\put(100,10){\textrm{boxes}}

\end{picture}

\

\

\

\

\

\

\

\

\

For the columns $c_1$, $c_2$, and $c_3$ we wish to show that $x_i < y_i$ for $i=1,2,3$. 
Since the rows were strictly increasing to begin with we have 
\[ x_1 < j \leq y_1, \textrm{ and}\]
\[ x_3 < j+i < y_3.\]
Further, for $k_1$ and $k_2$ shown, we have $x_2=j+k_2-k_1-1$ and $y_2=j+k_2$ hence
\[ x_2 \leq x_2 +k_1 = j+k_2-1 < y_2.\]

It can be seen that no two of these block rotations performed in \textbf{Step B} can move the same elements. 
If we have done a manipulation to fix a $j,j$ problem then there cannot be a second $j-1$ immediately left of the bottom-left element of this block since, otherwise, before rotating that block the row had the value $j-1$ repeated occuring on the bottom-left of the block. 
However, we know that the rows were strictly increasing to begin with and the only repeated values we have introduced were on the bottom-right of each block we have rotated. 
So since we are working right to left, this repeated value $j-1,j-1$ could not occur. 
Therefore, after rotating this block, there is no $j-1$ in the bottom row, and hence the next block's pair of consecutively increasing sequences cannot enter this block. 
Hence the two blocks are disjoint.

We continue this process of reading values and fixing future lattice problems by rotating blocks until there are no more future lattice problems left in the tableau. This results in a SSYT $T'$ of shape $\lambda$.

\

\noindent \textbf{Step C:} \textit{Extend from shape $\lambda$ to SSYT of shape $\mathcal{S}(\lambda,k,n)$} ($T' \rightarrow \mathcal{T}_1$)

\

Using $T'$ and the unique semistandard filling of $\Delta_n$ with lattice reading word, we obtain a SSYT $T_1$ of shape $\lambda \oplus \Delta_n$. 
We check that ${T}_1$ is a SSYT with a lattice reading word.
Since we removed all future lattice problems, $T'$ has the property that, when read, the number of $j+1$'s is always at most one more than the number of $j$'s. 
Thus the SSYT ${T}_1$ does have a lattice reading word. 
Therefore, applying Lemma~\ref{fatjoinlemma}, we obtain a SSYT $\mathcal{T}_1$ of shape $\mathcal{S}(\lambda,k,n)$ with lattice reading word.

\

\noindent \textbf{Check:} \textit{Can we recover $\mathcal{T}_2$?}

\

Finally, we must show that we can recover $\mathcal{T}_2$ from $\mathcal{T}_1$. 
Undoing \textbf{Step A} and \textbf{Step C} are trivial. 
Thus we need only show that the block rotations in \textbf{Step B} can be reversed. 
Since any two block rotations are disjoint, it suffices to check that a single block rotation can be undone. 
We have seen that a block rotation creates a tableau with a repeated value $j+i,j+i$ in the bottom-right of the block. 
Transposing the tableau, there is a repeated value in the second column and so the block rotation described in \textbf{Step 3} applies and it results in the transpose of the original tableau. 
Therefore these block rotations can be reversed, and we can recover $\mathcal{T}_2$ from $\mathcal{T}_1$.

Thus we have shown that $s_{\mathcal{S}(\lambda,k,n)}-s_{\mathcal{S}(\lambda^{t},k,n)} \geq_{s} 0$ as functions in $n+1$ variables. This and \textbf{1} completes the proof that $s_{\mathcal{S}(\lambda^{t},k,n)} =_{n+1} s_{\mathcal{S}(\lambda,k,n)}$.

\

\noindent \textbf{3.}

We now wish to show that the two skew Schur functions are distinct when there is at least $n+2$ variables. Hence we consider fillings using the values $\{1,2,\ldots, n+2 \}$. 

Let $l(\lambda)=2$ and $\lambda \neq \lambda^{t}$. 
The hypotheses of the theorem required $\lambda_1 >1$, so that $\lambda$ was not a single column. 
If $\lambda_1 =2$, then $\lambda$ is either the shape $(2,1)$ or $(2,2)$, which both have $\lambda^t = \lambda$. 

Thus we need only consider $\lambda_1 \geq 3$. 
We split into three cases depending on the size of $\lambda_2$. 
Namely,  

\begin{enumerate}
\item $\lambda_2=1$,
\item $\lambda_2=2$, and
\item $\lambda_2 \geq 3$.
\end{enumerate}

\noindent In each case we wish to create a SSYT $\mathcal{T}$ of shape $\mathcal{S}(\lambda^{t},k,n)$ with lattice reading word and content $\nu$ such that no SSYT of shape $\mathcal{S}(\lambda,k,n)$ with lattice reading word and content $\nu$ exists. 
Here we show the fillings of the foundation of the required tableau $\mathcal{T}$ for each of the above three cases.

\

\begin{tabular}{ccccccccccc}
&1.& &&&2.&  &&&3.&\\
&& && &&  &&&&\\
$2$ &  & $n+1$ && $2$ &  & $n+1$  && $2$ &  & $3$\\
$3$ &  &  && $3$ &  & $n+2$  && $3$ &  & $4$\\
$\vdots$ &  & && $\vdots$ &  &  && $\vdots$ &  &$\vdots$\\
$\lambda_1-2$ &  &  && $\lambda_1-2$ & &   && $\lambda_2-2$ &  & $\lambda_2-1$\\
$n$ &  & && $n$ &  &   && $\lambda_2-1$ &  & $n$\\
$n+1$ &&  && $n+1$ &&  &&$\lambda_2$ && $n+1$\\
$n+2$ & &  && $n+2$ & & && $\lambda_2+1$ && $n+2$ \\
&&&&&&&& $\vdots$  && \\
&&&&&&&&  $\lambda_1-2$ && \\
&&&&&&&&  $\lambda_1-1$ && \\
&&&&&&&&  $n+1$ && \\
&&&&&&&&  $n+2$ && 
\end{tabular}

\

If $\lambda_1 =3$ in case 1 or 2, then the first column consists of the three entries $n, n+1, n+2$, and if $\lambda_1 =3$ in case 3, then the first column consists of $2, n+1, n+2$. 
Similarly, if $\lambda_2 = 3$ in case 3, then the second column consists of $n, n+1, n+2$.

In each case it is clear that $\mathcal{T}$ is a SSYT with a lattice reading word. 
We now check that no SSYT of shape $\mathcal{S}(\lambda,k,n)$ with lattice reading word and the same content can exist. 
Any SSYT of shape $\mathcal{S}(\lambda,k,n)$ with lattice reading word can only have a single $n+1$ in the first row of $\lambda$ and a single $n+2$ in the second row of $\lambda$. Further, no $n+2$ can appear in the first row of $\lambda$. 
Thus, in Case 1, the $n+2$ must appear in the second row of $\lambda$, but then both $n+1$'s would have to appear in the first row, which is impossible. In Case 2 and Case 3, both $n+2$'s would have to fit in the second row of $\lambda$, which is impossible. Therefore $s_{\mathcal{S}(\lambda^{t},k,n)} \neq_{n+2} s_{\mathcal{S}(\lambda, k,n)}$. \qed

\end{proof}

\begin{example}

\

Here we illustrate a concrete example of \textbf{Step 3} and reversing \textbf{Step 3} in the proof of Theorem~\ref{2theorem}. Let us begin with the foundation of $\mathcal{S}((7,6), 0,n)$, where $n \geq 11$.

\

\begin{tabular}{ccccccc}
2&4&5&6&9&10&12    \\
6&7&7&8&11&11&

\end{tabular}

\

\noindent We now tranpose and perform \textbf{Step 3} twice. 
For convenience, in each step we highlight the block that is about to be rotated.

\

\begin{tabular}{cc ccc cc ccc cc}
2&\textbf{6}          &&               &  &2&5                      &&               & &2&5  \\
4&\textbf{7}          &&               &  &4&6                      &&               & &4&6  \\
\textbf{5}&\textbf{7} && $\rightarrow$ &  &6&7                      && $\rightarrow$ & &6&7  \\
\textbf{6}&8          &&               &  &7&8                      &&               & &7&8  \\
9&11                  &&               &  &9&\textbf{11}            &&               & &9&10 \\
10&11                 &&               &  &\textbf{10}&\textbf{11}  &&               & &11&11\\
12&                   &&               &  &12&                      &&               & &12&

\end{tabular}

\

\

To reverse \textbf{Step 3} to return to the original tableau we transpose and then either follow the rules for rotating discussed in the \textbf{Check} section of the proof of part \textbf{1} of Theorem~\ref{2theorem} or, equivalently, follow the the rules for rotating discussed in the \textbf{Step B} section of the proof of part \textbf{2} of Theorem~\ref{2theorem}.

\

\begin{tabular}{ccccccc}
2&4&6&7&9 &\textbf{11}&12    \\
5&6&7&8&\textbf{10}&\textbf{11}&     \\
&&&&&& \\
&&&$\downarrow$&&&\\
&&&&&&\\
2&4&\textbf{6}&\textbf{7}&9&10&12    \\
\textbf{5}&\textbf{6}&\textbf{7}&8&11&11&     \\
&&&&&& \\
&&&$\downarrow$&&&\\
&&&&&&\\
2&4&5&6&9&10&12    \\
6&7&7&8&11&11&     \\

\end{tabular}

\

\end{example}

\subsection{Hook Partitions}

We now inspect the case when the foundation is a hook diagram.

\begin{theorem}
\label{1hooktheorem}
If $\lambda=(\lambda_a,1^{\lambda_l - 1})$ is a hook with $\lambda_l \leq \lambda_a$ then for $0 \leq k \leq 1$ we have 
\begin{enumerate}
\item $s_{\mathcal{S}(\lambda^{t}, k,n)}-s_{\mathcal{S}(\lambda,k,n)} \geq_s 0,$ 
\item $s_{\mathcal{S}(\lambda^{t}, k,n)}=_{n+1}s_{\mathcal{S}(\lambda,k,n)},$ 
\item $s_{\mathcal{S}(\lambda^{t}, k,n)} \neq_{n+2} s_{\mathcal{S}(\lambda,k,n)}$ if $\lambda \neq \lambda^{t}$. 
\end{enumerate}
Furthermore, for $0 \leq k \leq 1$, we have 
\[ s_{\mathcal{S}(\lambda^{t}, k,n)}-s_{\mathcal{S}(\lambda,k,n)} = 
\sum_{B,C} \left[ \left( \begin{array}{c}
                  |B|-|C|-1 \\
                    \lambda_l -|C| -1\\
                    \end{array} \right)  
            -   \left( \begin{array}{c}
                  |B|-|C|-1 \\
                    \lambda_a -|C| -1\\
                    \end{array} \right)
         \right] s_{\nu(B,C)} \]
where $\nu(B,C) = \delta_n+ \sum_{b \in B} e_b + \sum_{c \in C} e_c + (0^{n+1},1^{|\lambda|-|B|-|C|})$ and
the sum is over all sets $B$, $C$ such that
\begin{itemize}
\item $n+1 \in B \subseteq \{ 2-k,3-k,\ldots,n+1 \}$, $C \subseteq \{ 3-k,4-k,\ldots,n+1 \}$,
\item $C \subset B$
\item $|B|+|C| \leq |\lambda|-1$, $|B| \geq \lambda_a$, $|C|+1 \leq \lambda_l$, and
\item if $C=\bigcup_{j=1}^{m} C_j$ where the $C_j$ are the maximal disjoint intervals of $C$, then $\textrm{min}(C_j)-1 \in B-C$ for each $j$.
\end{itemize}
\end{theorem}

\begin{proof} 

\

\noindent \textbf{1.}

Consider any SSYT $\mathcal{T}_1$ of shape $\mathcal{S}(\lambda,k,n)$ with content $\nu$ and lattice reading word.
We shall first show that there is a SSYT of shape $\mathcal{S}(\lambda^t,k,n)$ with content $\nu$ and lattice reading word.
Then, letting $\mathcal{S}(\lambda,k,n)=\kappa / \rho$ and $\mathcal{S}(\lambda^t,k,n)=\kappa_t / \rho_t$ for partitions $\kappa$, $\kappa_t$, $\rho$, and $\rho_t$, we shall show that the Littlewood-Richardson coefficients for these two diagrams and this content satisfy 
\[ c_{\rho_t \nu}^{\kappa_t} \geq c_{\rho \nu}^{\kappa}.\]
Having shown that this inequality holds for any content $\nu$ for which a SSYT of shape $\mathcal{S}(\lambda,k,n)$ with content $\nu$ and lattice reading word exists, this will show $s_{\mathcal{S}(\lambda^{t}, k,n)}-s_{\mathcal{S}(\lambda,k,n)} \geq_s 0.$

\

By Lemma~\ref{firstrowlem} we know that the first row of $\lambda$ contains a strictly increasing sequence $a_1<a_2<\ldots < a_{\lambda_a}$ where $a_{\lambda_a} \leq n+1$.
Also, since the columns of $\mathcal{T}_1$ strictly increase, the first column of $\lambda$ contains a strictly increasing sequence $a'_1<a'_2<\ldots < a'_{\lambda_l}$, where $a'_1=a_1$.

\

\noindent Given a particular content $\nu$, there are two cases that we must consider.

\begin{enumerate}
\item $l(\nu) \leq n+1$. 
\item $l(\nu) > n+1$. 
\end{enumerate}

\

\noindent \textbf{Case 1:}

Since $l(\nu) \leq n+1$, the largest possible entry of $\mathcal{T}_1$ is $n+1$. 
We obtain a tableau $T_2$ of shape $\lambda^t \oplus \Delta_n$ by simply transpose the foundation $\lambda$ of $\mathcal{T}_1$. 
Hence the first row $\lambda^t$ is the strictly increasing sequence $a'_1<a'_2<\ldots < a'_{\lambda_l}$. 
Also, if $k=0$ then $a_1=a'_1 >1$. Otherwise, if $k=1$, then $a_1=a'_1>0$.

It can also be seen that ${T}_2$ has a lattice reading word. 
Suppose not. 
Then, when reading the entries of $\lambda^t$, we must reach a point where we have read two $j$'s but no $j-1$'s. 
Since $\mathcal{T}_1$ has a lattice reading word and also has the two $j$'s in its foundation $\lambda$, a $j-1$ must appear in $\lambda$.
Thus a $j-1$ appears in $\lambda^t$ as well. 
Since the rows and columns of $\lambda^t$ strictly increase, the $j-1$ appears either in the first column of $\lambda^t$ above the entry $j$, or the $j-1$ appears in the first row of $\lambda^t$ left of the entry $j$. In either case the $j-1$ is read before the second $j$ is read, contradicting our assumption.

Thus we may apply Lemma~\ref{fatjoinlemma} to obtain a SSYT $\mathcal{T}_2$ of shape $\mathcal{S}(\lambda^t, k,n)$ with lattice reading word.

Further, we can recover $\mathcal{T}_1$ from $\mathcal{T}_2$ by transposition of the foundation of $\mathcal{T}_2$. 
This shows that if $l(\nu) \leq n+1$ then we have $c_{\rho_t \nu}^{\kappa_t} \geq c_{\rho \nu}^{\kappa}.$

\

\noindent \textbf{Case 2:}

If $l(\nu) > n+1$ then $n+2 \in \mathcal{T}_1$. 
Since $a_1<a_2<\ldots < a_{\lambda_a} \leq n+1$, then $n+2=a'_j$ for some $j$. 
It is clear that any SSYT of shape $\mathcal{S}(\lambda^t,k,n)$ with lattice reading word must have the values $a'_j, a'_{j+1},\ldots, a'_{\lambda_l}$ as the last $\lambda_1^t - j+1$ entries of the first column of $\lambda^t$. 
Since $\lambda_l - j+1 \leq \lambda_l \leq \lambda_a$, these entries do fit in this column.

Let $M$ be the multiset $\{a_1,a_2,\ldots,a_{\lambda_a}\} \cup \{a'_2,a'_3\ldots,a'_{j-1} \}$. 
Then $M$ is the remaining entries tha we need to place in $\lambda^t$. 
We have $|M|=\lambda_a+j-2$ and $\textrm{max}(M)=n+1$. 
Let $R =\{a_1,a_2,\ldots,a_{\lambda_a}\} \cap \{a'_1,a'_2\ldots,a'_{j-1} \}$.
We note that $R =\{a_1,a_2,\ldots,a_{\lambda_a}\} \cap \{a'_1,a'_2\ldots,a'_{\lambda_l} \}$ since $a_{\lambda_a} < n+2 = a'_{j}$. 
The set $R$ contains the values of $\mathcal{T}_1$ that appear in both the first row of $\lambda$ and the first column of $\lambda$.
For a SSYT of shape $\mathcal{S}(\lambda^t,k,n)$ with lattice reading word and content $\nu$ these values must also appear in both the first row of $\lambda^t$ and the first column of $\lambda^t$.

Consider $A=\{a_1,a_2,\ldots,a_{\lambda_a}\}-R$ and $A'=\{a'_1,a'_2\ldots,a'_{j-1} \}-R$.
Since we know that the values of $R$ must appear in both the first row of $\lambda^t$ and first column of $\lambda^t$, $A \cup A'$ contains the remaining values of $M$ that need to be placed in $\lambda^t$ 

Since $|R| < j \leq \lambda_l$, 
we can extend $R$ to an increasing sequence $b_1 < b_2 < \ldots < b_{\lambda_l}$ by choosing $\lambda_l - |R|$ additional values from $A \cup A'$. 
Then $M-\{b_1,\ldots,b_{\lambda_l} \} \subseteq M-R$ contains $k=|M|-\lambda_l=\lambda_a+j-2 -\lambda_l$ distinct values, each no greater than $n+1$.
That is, they are an increasing sequence $c_1 < c_2 < \ldots < c_{k}$, where $c_{k}\leq n+1$.
We have 
\begin{eqnarray*}
\  \lambda_a & = & 1+ (\lambda_a +j-2-\lambda_l) + (\lambda_l -j +1)
\\ & = & 1+k+(\lambda_l -j +1),
\\
\end{eqnarray*}
so we may fill $\lambda^t$ as shown below.

\

\begin{center}
\begin{tabular}{cccc}
$b_1$ & $b_2$ & $\cdots$ & $b_{\lambda_l}$ \\
$c_1$\\
$c_2$\\
$\vdots$\\
$c_{k}$\\
$a'_{j}$\\
$a'_{j+1}$\\
$\vdots$\\
$a'_{\lambda_l}$
\end{tabular}
\end{center}

\

\noindent This filling gives us the tableau ${T}_2$ of shape $\lambda^t \oplus \Delta_n$. 
Also $b_1>1$ for $k=0$ and $b_1>0$ for $k=1$.

It can also be seen that ${T}_2$ has a lattice reading word. 
Suppose not. 
Then, when reading the entries of $\lambda^t$, we must reach a point where we have read two $j$'s but no $j-1$'s. 
Since $\mathcal{T}_1$ has a lattice reading word and also has the two $j$'s in its foundation $\lambda$, a $j-1$ must appear in $\lambda$. Thus a $j-1$ appears in $\lambda^t$ as well.
Since the rows and columns of $\lambda^t$ strictly increase, either the $j-1$ appears either in the first column of $\lambda^t$ above the entry $j$, or the $j-1$ appears in the first row of $\lambda^t$ left of the entry $j$. In either case the $j-1$ is read before the second $j$ is read, contradicting our assumption.

Therefore we may apply Lemma~\ref{fatjoinlemma} to obtain a SSYT $\mathcal{T}_2$ of shape $\mathcal{S}(\lambda^t,k,n)$ with lattice reading word.

Therefore from any SSYT $\mathcal{T}_1$ of shape $\mathcal{S}(\lambda, k,n)$ with lattice reading word and content $\nu$ we can create a SSYT $\mathcal{T}_2$ of shape $\mathcal{S}(\lambda^{t}, k,n)$ with lattice reading word and content $\nu$.

\

Let $c_{\rho_t \nu}^{\kappa_t}$ be the number of SSYT of shape $\mathcal{S}(\lambda^{t}, k,n)$ with lattice reading word and content $\nu$, and $c_{\rho \nu}^{\kappa}$ be the number of SSYT of shape $\mathcal{S}(\lambda, k,n)$ with lattice reading word and content $\nu$.
We shall show that the sets $R$ and $A \cup A'$ described above are completely determined by $\nu$.

Since $\Delta_n$ is uniquely filled, from $c$ we can determine the content of the foundation $\lambda$ (or $\lambda^t$ respectively) needed to create a SSYT of shape $\mathcal{S}(\lambda, k,n)$ ($\mathcal{S}(\lambda^{t}, k,n)$ respectively) with lattice reading word and content $\nu$.
From the content of the foundation $\lambda$ ($\lambda^t$ respectively), we can determine the values $a'_j, a'_{j+1},\ldots, a'_{\lambda_l}$ greater than $n+1$ and we can determine the set $R$ of values that appear in both the first row of $\lambda$ ($\lambda^t$ resp.) and first column of $\lambda$ ($\lambda^t$ resp.).
Further, any SSYT of shape $\mathcal{S}(\lambda,k,n)$ ($\mathcal{S}(\lambda^t,k,n)$ resp.) must contain $R$ in the first row of $\lambda$ ($\lambda^t$ resp.) since otherwise the first column of $\lambda$ ($\lambda^t$ resp.) would not strictly increase.
After determining the rest of the entries of the first row of $\lambda$ ($\lambda^t$ respectively), the rest of the tableau is uniquely determined.
These entries are choses from the set $A \cup A' = \{ x \in \lambda \textrm{ } (\lambda^t \textrm{ resp.}) \textrm{ } | \textrm{ } x \leq n+1, \textrm{ } x \not\in R\}$.

Therefore the number of SSYTx of shape $\mathcal{S}(\lambda^{t}, k,n)= \kappa_t / \rho_t$ with lattice reading word and content $\nu$ is given by
\[ c_{\rho_t \nu}^{\kappa_t} = \left( \begin{array}{c}
                  |A \cup A'| \\
                    \lambda_l -|R|\\
           \end{array} \right)  \]
and the number of SSYTx of shape $\mathcal{S}(\lambda, k,n)= \kappa / \rho$ with lattice reading word and content $\nu$ is given by
\[ c_{\rho \nu}^{\kappa} = \left( \begin{array}{c}
                  |A \cup A'| \\
                    \lambda_a -|R|\\
           \end{array} \right).  \]

\

Since $\lambda_l \geq j-1$, for each $i$ we have  
\begin{eqnarray*}
\  \lambda_a -|R|-i & \geq & \lambda_a - |R|-i + (j-1 - \lambda_l)
\\ & \geq &  (\lambda_a -|R|) + (j-1 -|R|) - (\lambda_l -|R|) -i
\\ & \geq &   |A|+|A'|- (\lambda_l -|R|) -i.
\end{eqnarray*}
That is, 
\begin{equation}
\label{hooky}
\lambda_a -|R|-i \geq  |A \cup A'|- (\lambda_l -|R|) -i,
\end{equation}
for each $i$.

\

\noindent Therefore
\begin{eqnarray*}
\  c_{\rho_t \nu}^{\kappa_t} & = & \frac{(|A \cup A'|)! }{(|A \cup A'|-(\lambda_l-|R|))!(\lambda_l-|R|)!}
\\ & = & \frac{(|A \cup A'|)! }{(|A \cup A'|-(\lambda_a-|R|))!(\lambda_a-|R|)!} \times \prod_{i=0}^{\lambda_a - \lambda_l-1} \frac{\lambda_a - |R| -i}{|A \cup A'| -(\lambda_l - |R|) -i}
\\ & = & c_{\rho \nu}^{\kappa} \times \prod_{i=0}^{\lambda_a - \lambda_l-1} \frac{\lambda_a - |R| -i}{|A \cup A'| -(\lambda_l - |R|) -i}
\\ & \geq & c_{\rho \nu}^{\kappa} ,
\end{eqnarray*}
where we have used Equation~\ref{hooky} in the final step. Therefore $s_{\mathcal{S}(\lambda^{t}, k,n)}-s_{\mathcal{S}(\lambda,k,n)} \geq_s 0$ in this case as well. 

\

\noindent \textbf{2.}

We now show that we can reverse the process described in the proof of \textbf{1} when we restrict to fillings using the numbers $\{1,2,\ldots, n+1\}$.
If we take any SSYT $\mathcal{T}_2$ of shape $\mathcal{S}(\lambda^{t},k,n)$ with lattice reading word then the first column of $\lambda^{t}$ contains a strictly increasing sequence $a_1 < a_2 < \ldots < a_{\lambda_a}$ with $a_{\lambda_a} \leq n+1$, and the first row of $\lambda^t$ contains a strictly increasing sequence $a'_1 < a'_2 < \ldots < a'_{\lambda_l}$ with $a'_1 =a_1$ and $a'_{\lambda_l}\leq n+1$. 
We have $a_1=a'_1>1$ if $k=0$ and $a_1=a'_1>0$ if $k=1$.
By transposing $\lambda^t$ we obtain a SSYT $T_1$ of shape $\lambda \oplus \Delta_n$. 

Further, the tableau ${T}_1$ has a lattice reading word. Suppose not.
Then, when reading the entries of $\lambda$, we must reach a point where we have read two $j$'s but no $j-1$'s. 
Since $\mathcal{T}_2$ has a lattice reading word and also has the two $j$'s in its foundation $\lambda^t$, a $j-1$ must appear in $\lambda^t$. Thus $\lambda$ contains a $j-1$ as well.
Since the rows and columns of $\lambda$ strictly increase, either the $j-1$ appears either in the first column of $\lambda$ above the entry $j$, or the $j-1$ appears in the first row of $\lambda$ left of the entry $j$. In either case the $j-1$ is read before the second $j$ is read, contradicting our assumption.

Therefore we may apply Lemma~\ref{fatjoinlemma} to obtain a SSYT of shape $\mathcal{S}(\lambda, k,n)$ with lattice reading word. 
Finally, we can recover the tableau $\mathcal{T}_2$ from $\mathcal{T}_1$ by transposing the entries of the foundation of $\mathcal{T}_1$.
This shows that $s_{\mathcal{S}(\lambda, k,n)}-s_{\mathcal{S}(\lambda^{t},k,n)} \geq_s 0$ in $n+1$ variables. This and \textbf{1.} completes the proof that $s_{\mathcal{S}(\lambda^{t}, k,n)} =_{n+1} s_{\mathcal{S}(\lambda,k,n)}$.

\

\noindent \textbf{3.}

If $\lambda \neq \lambda^{t}$ then $\lambda_a \geq \lambda_l - 1$ and $\lambda_a \geq 3$. 
Consider the following tableau $T$ of shape $\lambda^t$.

\

\begin{center}
\begin{tabular}{ccccccc}
$1$ & $2$ & $\cdots$ & $\lambda_l-2$ & $n$ & $n+1$\\
$2$\\
$\vdots$\\
$\lambda_a-3$\\
$n$\\
$n+1$\\
$n+2$
\end{tabular}
\end{center}

\

If $\lambda_a =3$ then the first column of $\lambda^t$ consists of the entries $n,n+1,n+2$, and if $\lambda_l =2$ then the first row of $\lambda^t$ consists of the entries $n,n+1$. 
For either $k=0$ or $k=1$, this tableau clearly extends to a SSYT $\mathcal{T}_1$ of shape $\mathcal{S}(\lambda^t,k,n)$ with lattice reading word.
Further, since $\lambda_a \geq \lambda_l - 1$, both the first row of $T$ and the first column of $T$ contain the set of elements $ R=\{1,2,\ldots, \lambda_l-2, n, n+1 \}$ in common.

Suppose we are given a SSYT $\mathcal{T}_2$ of shape $\mathcal{S}(\lambda, k, n)$ with lattice reading word and the same content as $\mathcal{T}_1$. 
Then the foundation of $\mathcal{T}_2$ must also contain the set of elements $R$ in its first row and its first column.
But the foundation of $\mathcal{T}_2$ is of shape $\lambda$; thus, the first column of the foundation of $\mathcal{T}_2$ is precisely the elements $R=\{1,2,\ldots, \lambda_l-2, n, n+1 \}$. This leaves no room for $n+2$, since it cannot appear in the first $n+1$ rows of $\mathcal{T}_2$. Therefore no such tableau $\mathcal{T}_2$ can exist.
Therefore $s_{\mathcal{S}(\lambda, k,n)} \neq_{n+2} s_{\mathcal{S}(\lambda^{t},k,n)}$. 

\

This completes the proof of \textbf{1}, \textbf{2}, and \textbf{3}. We now show that the formula stated for $s_{\mathcal{S}(\lambda^{t}, k,n)} - s_{\mathcal{S}(\lambda,k,n)}$ is correct.

In the proof of \textbf{1}, we saw that if $l(\nu) \leq n+1$ then the number of SSYT of shape $\mathcal{S}(\lambda^{t}, k,n)$ with lattice reading word and content $\nu$ is the same as the number of SSYT of shape $\mathcal{S}(\lambda, k,n)$ with lattice reading word and content $\nu$.
Then we saw that if $l(\nu) \geq n+2$, then a SSYT of shape $\mathcal{S}(\lambda^{t}, k,n)$ with lattice reading word and content $\nu$ exists precisely when a SSYT of shape $\mathcal{S}(\lambda, k,n)$ with lattice reading word and content $\nu$ exists.
Furthermore, when such tableaux exist, we saw that the number of SSYT of shape $\mathcal{S}(\lambda^{t}, k,n)= \kappa_t / \rho_t$ with lattice reading word and content $\nu$ is given by
\[ c_{\rho_t \nu}^{\kappa_t} = \left( \begin{array}{c}
                  |A \cup A'| \\
                    \lambda_l -|R|\\
           \end{array} \right)  \]
and the number of SSYT of shape $\mathcal{S}(\lambda, k,n)= \kappa / \rho$ with lattice reading word and content $\nu$ is given by
\[ c_{\rho \nu}^{\kappa} = \left( \begin{array}{c}
                  |A \cup A'| \\
                    \lambda_a -|R|\\
           \end{array} \right),  \]
where, for the given content $\nu$, $R$ was the set of values that must appear in both the first row and first column of the foundation of each diagram, and $A \cup A'$ are the values $\leq n+1$ not in $R$ that must also appear in the foundation of each diagram.
We recall that all but one element of $R$ appear twice in the foundation. The smallest value of $R$ being the top-left entry of each foundation.

Let $\nu$, $l(\nu) \geq n+2$, be such a content.
If we let $C$ be the values that appear twice in the foundation we obtain $|R|=|C|+1$.
Further, if we let $B$ be the set of all values $\leq n+1$ that appear in the foundation, we have $A \cup A'= B-R$, so that $|A \cup A'| = |B|-|R|=|B|-|C|-1.$

Thus for each content $\nu$ we obtain a pair $B,C$.
We now show that the sets $B$ and $C$ have the properties listed in the statement of the theorem. 
We have $B \subseteq \{1,2,\ldots n+1\}$ when $k=1$ and $B \subseteq \{2,3\ldots n+1\}$ when $k=0$.
Thus $B \subseteq \{2-k,3-k,\ldots n+1\}$.

Decompose $C=\bigcup_{j=1}^{m}C_j$ for disjoint $C_j$. 
Then, by the definition of $C$, for each $C_j$, $\textrm{min}(C_j)$ appears twice in the foundation but $\textrm{min}(C_j) -1$ does not.
The lattice condition then implies that $\textrm{min}(C_j) -1 \in B-C$ for each $j$.
In particular $\textrm{min}(C) -1 \in B-C$.
This and $C \subseteq B \subseteq \{2-k,3-k,\ldots n+1\}$ gives $C \subseteq \{3-k,4-k,\ldots n+1\}$.

Additionally, we have $|B|+|C| \leq |\lambda|-1$ since we assumed $l(\nu) \geq n+2$.
Since the entries of the first row are distinct, each is an element of $B$. Thus $|B| \geq \lambda_a$. 
Also, since $|R|$ appears in the first column of $\lambda$ we have $|C|+1 = |R| \leq \lambda_l$.
Therefore $B$ and $C$ have the desired properties.

\

Conversely, for any sets $B$ and $C$ with the desired properties we show that there exists SSYT of shape $\mathcal{S}(\lambda,k,n)$ with lattice reading word and content
\[ \nu(B,C) = \delta_n+ \sum_{b \in B} e_b + \sum_{c \in C} e_c + (0^{n+1},1^{|\lambda|-|B|-|C|}).\]

To this end we create a filling of $\lambda$ with content $\sum_{b \in B} e_b + \sum_{c \in C} e_c + (0^{n+1},1^{|\lambda|-|B|-|C|})$.
We first place the $|\lambda|-|B|-|C|$ values $n+2,n+3,\ldots,n+|\lambda|-|B|-|C|+1$ at the bottom of the first column of $\lambda$.
We then place $\textrm{min}(B)$ at the top of the first column of $\lambda$.
Then we choose $\lambda_a-|C|-1$ values from $B-C-\textrm{min}(B)$.
We place these values together with the values in $C$ in increasing order in the remaining $\lambda_a-1$ positions in the first row of $\lambda$.
We place the remaining values of $B-C-\textrm{min}(B)$ together with the values in $C$ in increasing order in the remaining positions of the first column of $\lambda$.
This gives a semistandard filling of $\lambda$.
Together with the unique filling of $\Delta_n$, this gives a SSYT $T$ of shape $\mathcal{S}(\lambda,k,n)$ and content $\nu(B,C)$.

We now show that $T$ has a lattice reading word.
If $h \leq n+1$ appears twice in $\lambda$, then $h \in C$. Hence $h \in C_j$ for some $j$. 
Then either $h=\textrm{min}(C_j)$, in which case $h-1 \in B-C$ appears once in $\lambda$,
or $h \neq \textrm{min}(C_j)$, in which case $h-1 \in C$ appears twice in $\lambda$.
Thus, if $h \leq n+1$ appears in $\lambda$, then the number of $h$'s is $\lambda$ is at most one less than the number of $h-1$'s in $\lambda$.
Further, if $h \geq n+2$ appears in $\lambda$, then each value $n+2,n+3,\ldots,h-1$ appears in $\lambda$.
Thus, for all $h$, if $h$ appears in $\lambda$, then the number of $h$'s is $\lambda$ is at most one less than the number of $h-1$'s in $\lambda$.
Together with the content of $\Delta_n$, this shows that $T$ has a lattice reading word.

\

Therefore the terms $s_\nu$ appearing in the difference $s_{\mathcal{S}(\lambda^{t}, k,n)} - s_{\mathcal{S}(\lambda,k,n)}$ are precisely those $\nu$ of the form $\nu(B,C)$, where the sets $B$ and $C$ have the listed properties.
Now, since we have that \[ c_{\rho_t \nu(B,C)}^{\kappa_t} = \left( \begin{array}{c}
                  |A \cup A'| \\
                    \lambda_l -|R|\\
                    \end{array} \right) 
  =  \left( \begin{array}{c}
                  |B|-|C|-1 \\
                    \lambda_l -|C| -1\\
                    \end{array} \right)  \]
and \[ c_{\rho \nu(B,C)}^{\kappa} = \left( \begin{array}{c}
                  |A \cup A'| \\
                    \lambda_a -|R|\\
                    \end{array} \right) 
  =  \left( \begin{array}{c}
                  |B|-|C|-1 \\
                    \lambda_a -|C| -1\\
                    \end{array} \right),  \]
\noindent we find
\[ s_{\mathcal{S}(\lambda^{t}, k,n)} - s_{\mathcal{S}(\lambda,k,n)} 
   = \sum_{B,C} \left[ 
              \left( \begin{array}{c}
                  |B|-|C|-1 \\
                    \lambda_l -|C| -1\\
                    \end{array} \right)  
            -   \left( \begin{array}{c}
                  |B|-|C|-1 \\
                    \lambda_a -|C| -1\\
                    \end{array} \right)  
             \right]  s_{\nu(B,C)},  \\     \]
where the sum is over all $B$,$C$ with the listed properties. \qed
\end{proof}

\

\begin{example}
Let $n=4$, $\lambda = (3,1)$, and $k=0$.
We are interested in the following two diagrams.

\setlength{\unitlength}{0.4mm}

\begin{picture}(0,100)(-30,10)

\put(70,90){\framebox(10,10)[tl]{ }}

\put(60,80){\framebox(10,10)[tl]{ }}
\put(70,80){\framebox(10,10)[tl]{ }}

\put(50,70){\framebox(10,10)[tl]{ }}
\put(60,70){\framebox(10,10)[tl]{ }}
\put(70,70){\framebox(10,10)[tl]{ }}

\put(40,60){\framebox(10,10)[tl]{ }}
\put(50,60){\framebox(10,10)[tl]{ }}
\put(60,60){\framebox(10,10)[tl]{ }}
\put(70,60){\framebox(10,10)[tl]{ }}

\put(40,50){\framebox(10,10)[tl]{ }} \put(50,50){\framebox(10,10)[tl]{ }}
\put(40,40){\framebox(10,10)[tl]{ }}
\put(40,30){\framebox(10,10)[tl]{ }}

\put(40,15){$\mathcal{S}( \lambda^t , 0,4)$}
\put(140,15){$\mathcal{S}( \lambda , 0,4)$}

\put(170,90){\framebox(10,10)[tl]{ }}

\put(160,80){\framebox(10,10)[tl]{ }}
\put(170,80){\framebox(10,10)[tl]{ }}

\put(150,70){\framebox(10,10)[tl]{ }}
\put(160,70){\framebox(10,10)[tl]{ }}
\put(170,70){\framebox(10,10)[tl]{ }}

\put(140,60){\framebox(10,10)[tl]{ }}
\put(150,60){\framebox(10,10)[tl]{ }}
\put(160,60){\framebox(10,10)[tl]{ }}
\put(170,60){\framebox(10,10)[tl]{ }}

\put(140,50){\framebox(10,10)[tl]{ }} \put(140,40){\framebox(10,10)[tl]{ }}
\put(150,50){\framebox(10,10)[tl]{ }}
\put(160,50){\framebox(10,10)[tl]{ }}

\end{picture}

\noindent Theorem~\ref{1hooktheorem} gives
 \[s_{\mathcal{S}( \lambda^t , 0,4)}-s_{\mathcal{S}( \lambda,0,4)} = \sum_{B,C}   s_{\nu(B,C)},\]
where $5 \in B \subseteq \{ 2,3,4,5 \}$, $C \subseteq \{ 3,4,5 \}$, $C \subset B$, 
$|B|+|C| \leq 3$, $|B|\geq 2$, $|C| \leq 1$, and $\textrm{min}(C)-1 \in B-C$.

The condition $|C| \leq 1$ implies that $C$ is empty, or $C$ is a singleton.
If $C = \{c\}$ is a singleton then $|B|+|C| \leq 3$ gives that $|B| \leq 2$.
Since $C \subset B$ and $\textrm{min}(C)-1 \in B$, we have $B = \{c-1,c \}$.
Since $5 \in B$ and $C \subseteq \{ 3,4,5 \}$ this means that $C=\{5\}$ is the only possible singleton and for this set the $B$ satisfing the rest of the conditions is $B=\{4,5\}$.
The coefficient of $\nu(\{4,5\},\{5\})$ is  
\[ \left( \begin{array}{c}
                  2-1-1 \\
                    2 -1 -1\\
                    \end{array} \right)  
            -   \left( \begin{array}{c}
                  2-1-1 \\
                    3 -1 -1\\
                    \end{array} \right) = 
\left( \begin{array}{c}
                  0 \\
                    0\\
                    \end{array} \right)  
            -   \left( \begin{array}{c}
                  0 \\
                    1\\
                    \end{array} \right) = 1-0=1.\]

If $C= \emptyset$, then $|B|+|C| \leq 3$ gives that $|B| \leq 3$. Hence $2 \leq |B| \leq 3$ and $5 \in B \subseteq \{ 2,3,4,5 \}$.
For $|B|=2$ we have the possibilities $B=\{2,5\}$, $B=\{3,5\}$, and $B=\{4,5\}$.
For these $B$, the coefficient of $\nu(B, \emptyset)$ is
\[ \left( \begin{array}{c}
                  2-0-1 \\
                    2 -0 -1\\
                    \end{array} \right)  
            -   \left( \begin{array}{c}
                  2-0-1 \\
                    3 -0 -1\\
                    \end{array} \right) = 
\left( \begin{array}{c}
                  1 \\
                    1\\
                    \end{array} \right)  
            -   \left( \begin{array}{c}
                  1 \\
                    2\\
                    \end{array} \right) = 1-0=1.\]

For $|B|=3$ we have the possibilities $B=\{2,3,5\}$,$B=\{2,4,5\}$, and $B=\{3,4,5\}$.
For these $B$, the coefficient of $\nu(B, \emptyset)$ is
\[ \left( \begin{array}{c}
                  3-0-1 \\
                    2 -0 -1\\
                    \end{array} \right)  
            -   \left( \begin{array}{c}
                  3-0-1 \\
                    3 -0 -1\\
                    \end{array} \right) = 
\left( \begin{array}{c}
                  2 \\
                    1\\
                    \end{array} \right)  
            -   \left( \begin{array}{c}
                  2 \\
                    2\\
                    \end{array} \right) = 2-1=1.\]

Thus we have
\begin{eqnarray*}
s_{\mathcal{S}( \lambda^t , 0,4)}-s_{\mathcal{S}( \lambda,0,4)} &=& \sum_{B,C}   s_{\nu(B,C)} \\
&=& s_{\nu(\{4,5\}, \{5\})} +  s_{\nu(\{2,5\},\emptyset)} +  s_{\nu(\{3,5\},\emptyset)} +  s_{\nu(\{4,5\},\emptyset)} \\
&\textrm{ } & +  s_{\nu(\{2,3,5\},\emptyset)} +  s_{\nu(\{2,4,5\},\emptyset)} +  s_{\nu(\{3,4,5\},\emptyset)} \\
&=& s_{(4,3,2,2,2,1)} + s_{(4,4,2,1,1,1,1)} + s_{(4,3,3,1,1,1,1)} + s_{(4,3,2,2,1,1,1)}  \\
&\textrm{ } & + s_{(4,4,3,1,1,1)} + s_{(4,4,2,2,1,1)} + s_{(4,3,3,2,1,1)},   \\
\end{eqnarray*}
where we have omitted the details of computing $\nu(B,C)$ for each $B,C$.
For an example of this, we have
\begin{eqnarray*}
\nu(\{4,5\}, \{5\}) &=& \delta_4 + \sum_{b \in \{4,5 \} } e_b + \sum_{c \in \{5 \} } e_c + (0^5,1^{4-2-1} ) \\
&=& (4,3,2,1) + (0,0,0,1,1) + (0,0,0,0,1) + (0,0,0,0,0,1)\\
&=& (4,3,2,2,2,1). \\
\end{eqnarray*}

\end{example}

\section{A Finite Variable Conjecture}

In this section we are concerned with equalities of skew Schur functions in finitely many variables.

In Theorem~\ref{1theorem} and Theorem~\ref{2theorem} we proved that for each $n$ and $0 \leq k \leq 1$, we have 
\[ s_{ \mathcal{S}(\lambda^t,k,n)} =_{n+1} s_{\mathcal{S}(\lambda,k,n)} \] when $\lambda$ has either one or two parts or when $\lambda$ is a hook.

Since we have proven the finite variable equality
\[ s_{ \mathcal{S}(\lambda^t,k,n)} =_{n+1} s_{\mathcal{S}(\lambda,k,n)} \] 
for these simple shapes $\lambda$, one might ask for which other partitions $\lambda$ is this equality true.

We make the following conjecture.

\begin{conjecture}
\label{finstairconj}
For each $n \geq 1$, $k \geq 0$, and partition $\lambda$ such that $\mathcal{S}(\lambda,k,n)$ and $\mathcal{S}(\lambda^t,k,n)$ are connected skew diagrams we have
\[ s_{ \mathcal{S}(\lambda^t,k,n)} =_{n+1} s_{\mathcal{S}(\lambda,k,n)}. \]

\end{conjecture}

For each pair $n$, $k$, there are many partitions $\lambda$ such that both $\mathcal{S}(\lambda,k,n)$ and $\mathcal{S}(\lambda^t,k,n)$ are defined. 
Namely, for all partitions $\lambda$ contained in the square $({n+k}^{n+k})$. 

\

\

\

\

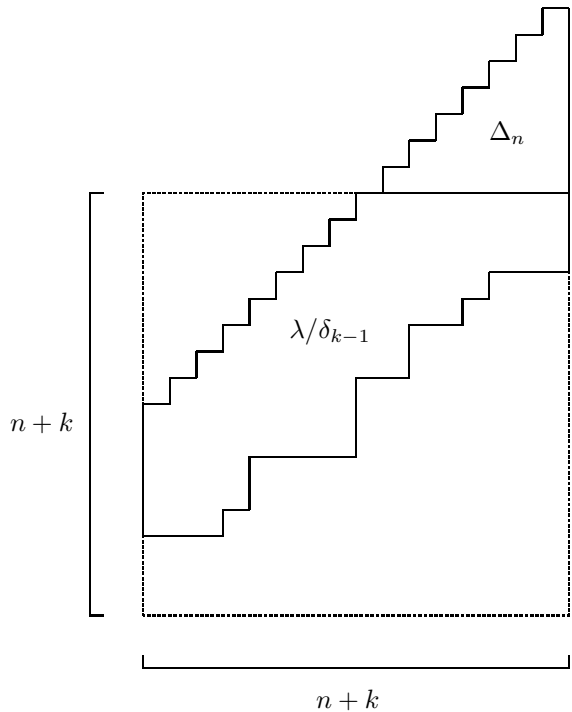
\setlength{\unitlength}{0.35mm}

\begin{picture}(100,60)(-100,40)

\put(-30,-110){\dashbox{1}(160,160)[tl]{ }}

\put(50,50){\line(1,0){80}}

\put(130,50){\line(0,1){70}}

\put(-30,-30){\line(1,0){10}}
\put(-20,-20){\line(1,0){10}}
\put(-10,-10){\line(1,0){10}}
\put(0,0){\line(1,0){10}}
\put(10,10){\line(1,0){10}}
\put(20,20){\line(1,0){10}}
\put(30,30){\line(1,0){10}}
\put(40,40){\line(1,0){10}}

\put(50,50){\line(1,0){10}}
\put(60,60){\line(1,0){10}}
\put(70,70){\line(1,0){10}}
\put(80,80){\line(1,0){10}}
\put(90,90){\line(1,0){10}}
\put(100,100){\line(1,0){10}}
\put(110,110){\line(1,0){10}}
\put(120,120){\line(1,0){10}}

\put(-20,-30){\line(0,1){10}}
\put(-10,-20){\line(0,1){10}}
\put(0,-10){\line(0,1){10}}
\put(10,0){\line(0,1){10}}
\put(20,10){\line(0,1){10}}
\put(30,20){\line(0,1){10}}
\put(40,30){\line(0,1){10}}
\put(50,40){\line(0,1){10}}

\put(60,50){\line(0,1){10}}
\put(70,60){\line(0,1){10}}
\put(80,70){\line(0,1){10}}
\put(90,80){\line(0,1){10}}
\put(100,90){\line(0,1){10}}
\put(110,100){\line(0,1){10}}
\put(120,110){\line(0,1){10}}

\put(-30,-80){\line(1,0){30}}
\put(0,-70){\line(1,0){10}}
\put(10,-50){\line(1,0){40}}
\put(50,-20){\line(1,0){20}}
\put(70,0){\line(1,0){20}}
\put(90,10){\line(1,0){10}}
\put(100,20){\line(1,0){30}}

\put(-30,-80){\line(0,1){50}}
\put(0,-80){\line(0,1){10}}
\put(10,-70){\line(0,1){20}}
\put(50,-50){\line(0,1){30}}
\put(70,-20){\line(0,1){20}}
\put(90,0){\line(0,1){10}}
\put(100,10){\line(0,1){10}}
\put(130,20){\line(0,1){30}}

\put(25,-5){$\lambda / \delta_{k-1}$}

\put(100,70){$\Delta_n$}

\put(-80,-40){$n+k$}

\put(-50,-110){\line(0,1){160}}
\put(-50,-110){\line(1,0){5}}
\put(-50,50){\line(1,0){5}}

\put(-30,-130){\line(1,0){160}}
\put(-30,-130){\line(0,1){5}}
\put(130,-130){\line(0,1){5}}

\put(35,-145){$n+k$}

\end{picture}

\

\

\

\

\

\

\

\

\

\

\

\

\

\

\

\

\

\noindent Both will be connected skew diagrams so long as $\delta_{k+1} \subseteq \lambda \subseteq ({n+k}^{n+k})$.

Computer-assisted calculations performed using John Stembridge's ``SF-package for maple" \cite{stem} have shown that Conjecture~\ref{finstairconj} holds for the following values of the pair $n,k$, when searching through all pairs of connected skew shapes $\mathcal{S}(\lambda^t,k,n)$, $\mathcal{S}(\lambda^t,k,n)$. 
\begin{itemize}
\item $n=2$, $0 \leq k \leq 5$,
\item $n=3$, $0 \leq k \leq 4$, and
\item $n=4$, $0 \leq k \leq 1$.
\end{itemize}
When $n=1$, the conjecture gives an equality in two variables. 
In this case it is not hard to check that the only pairs of skew shapes $\mathcal{S}(\lambda^t,k,n)$, $\mathcal{S}(\lambda^t,k,n)$ with \textit{nonzero} skew Schur functions in two variables have $\lambda = \lambda^t$, so the conjecture does hold for $n=1$.

As $n$ and $k$ increase, the number of shapes $\mathcal{S}(\lambda,k,n)$ grows quickly. 
This search space can be reduced by only considering the partitions $\lambda$ such that $\lambda$ is contained in $\delta_{n+k}$, since 
otherwise both skew diagrams $\mathcal{S}(\lambda,k,n)$ and $\mathcal{S}(\lambda^t,k,n)$ contain a column of length $n+2$, which implies that both skew Schur functions equal $0$ in $n+1$ variables. 
Nevertheless, as $n$ and $k$ increase the average size of the shapes increase, which makes it very time-consuming to compute each skew Schur function $s_{\mathcal{S}(\lambda,k,n)}$. 
This makes computationally verifying Conjecture~\ref{finstairconj} for a given pair $n$, $k$ difficult for larger values of $n$ and $k$.

\end{document}